\documentclass{amsart}
\usepackage{amssymb, delimset, booktabs, keytheorems, etoolbox}
\usepackage{mathtools}

\DeclarePairedDelimiter\floor{\lfloor}{\rfloor}
\usepackage{graphicx}
\usepackage{tabularx}
\usepackage{enumitem}
\setlist{itemsep=4pt, topsep=0pt, leftmargin=17pt}

\usepackage{floatrow}
\usepackage{xcolor}
\definecolor{BIT}{cmyk}{1, 0, 1, 0}

\usepackage[numbers, sort&compress, nonamebreak, merge, elide, longnamesfirst]{natbib}

\usepackage{hyphenat}
\makeatletter
\patchcmd{\NAT@citexnum}
{\@citea \NAT@test{\@ne}\NAT@spacechar\NAT@mbox{\NAT@super@kern\NAT@@open}}
{\@citea \nohyphens{\NAT@hyper@{\NAT@test{\@ne}}}\nobreakspace{}\NAT@mbox{\NAT@super@kern\NAT@@open}}
{}
{\PackageError{nonbreaking-citet}
{Failed to patch \string\citet}
{The internal definition of natbib may have changed.}}
\makeatother

\usepackage{tikz}
\usetikzlibrary{decorations.pathreplacing, calc, positioning}
\tikzset{
edge/.style={semithick},
ball/.style={shape=circle, minimum size=1mm, ball color=black, inner sep=0.5},
ellipsis/.style={shape=circle, fill, inner sep=.5}}

\usepackage[pagebackref,
colorlinks=true,
linkcolor=blue,
citecolor=BIT,
urlcolor=blue,
linktoc=all,
bookmarks=true]{hyperref}

\usepackage[capitalize, sort&compress]{cleveref}
\crefname{equation}{Eq.}{Eqs.}
\Crefname{equation}{Eq.}{Eqs.}
\crefname{table}{Tab.}{Tabs.}
\Crefname{table}{Tab.}{Tabs.}
\crefname{figure}{Fig.}{Figs.}
\Crefname{figure}{Fig.}{Figs.}
\crefalias{subsection}{section}
\crefformat{pieri}{#2Pieri's rule#3}
\Crefformat{pieri}{#2Pieri's rule#3}
\crefformat{schuralphaomega}{the #2Schur alpha--omega lemma#3}
\Crefformat{schuralphaomega}{The #2Schur alpha--omega lemma#3}
\crefformat{dominancematching}{#2the dominance-matching criteria#3}
\Crefformat{dominancematching}{#2The dominance-matching criteria#3}
\crefname{itm}{Condition}{Conditions}
\creflabelformat{itm}{~\upshape#2#1#3}
\crefname{def}{Def.}{Defs.}
\Crefname{def}{Def.}{Defs.}
\creflabelformat{def}{~\upshape(#2#1#3)}
\crefname{ineq}{Ineq.}{Ineqs.}
\Crefname{ineq}{Ineq.}{Ineqs.}
\creflabelformat{ineq}{~\upshape(#2#1#3)}
\crefname{conjecture}{Conjecture}{Conjectures}
\crefname{appendix}{Appendix}{Appendices}
\crefformat{section}{#2\S#1#3}
\Crefformat{section}{#2Section~#1#3}
\Crefrangeformat{section}{#3Sections~#1#4--#5#2#6}
\Crefmultiformat{section}
{#2Sections~#1#3}
{ and~#2#1#3}
{, #2#1#3}
{, and~#2#1#3}
\crefrangeformat{section}{#3\S#1#4--#5\S#2#6}
\crefmultiformat{section}
{#2\S#1#3}
{ and~#2\S#1#3}
{, #2\S#1#3}
{, and~#2\S#1#3}

\AddToHook{env/conjecture/begin}{\crefalias{theorem}{conjecture}}
\AddToHook{env/corollary/begin}{\crefalias{theorem}{corollary}}
\AddToHook{env/lemma/begin}{\crefalias{theorem}{lemma}}
\AddToHook{env/proposition/begin}{\crefalias{theorem}{proposition}}

\makeatletter
\@ifundefined{@kernel@get@ctr@root}{}{%
\let\csf@cref@constructprefix\cref@constructprefix
\def\cref@constructprefix#1#2{%
\edef\csf@cref@counter{\@kernel@get@ctr@root{#1}}%
\expandafter\csf@cref@constructprefix\expandafter{\csf@cref@counter}{#2}}}
\newcommand\creflabel[2][\@currentcounter]{%
\crefalias{\@currentcounter}{#1}\label{#2}}
\makeatother

\newtheorem{theorem}{Theorem}[section]
\newkeytheorem{corollary}[name=Corollary, sibling=theorem]
\newkeytheorem{conjecture}[name=Conjecture, sibling=theorem]
\newkeytheorem{lemma}[name=Lemma, sibling=theorem]
\newkeytheorem{proposition}[name=Proposition, sibling=theorem]
\newkeytheorem{problem}[name=Problem, sibling=theorem]

\theoremstyle{definition}

\newtheorem{example}[theorem]{Example}

\numberwithin{equation}{section}
\numberwithin{table}{section}
\numberwithin{figure}{section}

\newcommand{\coeff}[2]{\brk[s]!{#1}#2}
\newcommand{\spos}{\succeq_s 0}

\allowbreak
\allowdisplaybreaks

\title[Schur positivity from signed elementary expansions]
{Schur positivity from signed elementary expansions: clique-spiders and spiders $S(a,b,2)$}

\author[D.G.L.~Wang]{David G.L. Wang}
\thanks{The work of David G.L. Wang was supported by the General Program of the
National Natural Science Foundation of China (Grant No.~12171034).}

\author[W.Z.Y.~Wang]{Watson Z.Y. Wang\textsuperscript{*}}
\address{School of Mathematics and Statistics, Beijing Institute of Technology,
Beijing 102400, China}
\email[D.G.L. Wang]{glw@bit.edu.cn}
\email[W.Z.Y. Wang]{watsonzywang@bit.edu.cn}
\thanks{\textsuperscript{*}Corresponding author.}

\subjclass[2020]{05E05, 05C15, 05A15}
\keywords{chromatic symmetric function, Schur positivity, $e$-positivity,
spider-conjoined graph, dominance order, Hall's marriage theorem}
\date{}

\hypersetup{
pdftitle={Schur positivity from signed elementary expansions: clique-spiders and spiders S(a,b,2)},
pdfauthor={David G.L. Wang and Watson Z.Y. Wang},
pdfkeywords={chromatic symmetric function, Schur positivity, e-positivity, spider-conjoined graph, spider},
linktoc=all
}

\begin{document}
\begin{abstract}
Determining when the chromatic symmetric function of a graph is Schur positive is a central problem in algebraic combinatorics. Elementary expansions can help answer this question, but negative coefficients obstruct a direct positivity argument. We develop methods for proving Schur positivity from signed elementary expansions and apply them to chromatic symmetric functions of clique-spiders and three-legged spiders. We first give the Schur alpha--omega lemma, which bounds Schur support. We then give four equivalent sufficient criteria for Schur positivity from a fixed composition-indexed elementary expansion. The matrix, order-ideal, negative Hall, and positive Hall formulations treat positive coefficients as capacities and negative coefficients as demands, with admissible assignments governed by dominance order. These criteria provide certificates of Schur positivity especially when elementary coefficients are negative. For the four basic families of $3$-clique-spiders with total leg length at most one, we obtain complete classifications of both $e$-positivity and Schur positivity. These classifications show how the position of a single leg edge changes positivity. Using selected clique-spider cases and a path--clique bootstrap, we prove that all spiders of the form~$S(a,b,2)$ are Schur positive.
\end{abstract}
\maketitle

\section{Introduction}

\citet{Sta95} introduced the chromatic symmetric function~$X_G$ as a
generalization of Birkhoff's chromatic polynomial in the study of the Four Color
Problem that records proper colorings by color multiplicities.
For any basis $b=\{b_\lambda\}_\lambda$ of the algebra of symmetric functions,
a graph is said to be \emph{$b$-positive} if every~$b$-coefficient of~$X_G$ is
nonnegative.
Two central positivity questions ask when this symmetric function is positive in
the elementary or Schur basis.
For a symmetric function~$F$,
we write $F\spos$ if all its Schur coefficients are nonnegative,
and $F\not\spos$ if at least one of its Schur coefficients is negative.
For a graph~$G$,
we use the convention that $G\spos$ means $X_G\spos$,
and $G\not\spos$ means $X_G\not\spos$.
Schur positivity is representation-theoretically natural,
since Schur functions encode the irreducible characters of symmetric groups;
see \citet[Chap.~I]{Mac95B}.
The~$e$-positivity is stronger than Schur positivity,
since $e_\lambda\spos$ for every partition~$\lambda$.

The \emph{Stanley--Stembridge conjecture} \citep{SS93} asserted that the
incomparability graph of every $(3+1)$-free poset is~$e$-positive.
Recently,
\citet{Hik24X} proved the Stanley--Stembridge conjecture;
see also \citet{GMRWW25X,HHKKO25X,Vai26X}. On the other hand,
the Stanley--Stembridge conjecture can be equivalently reduced to that every
claw-free interval graph is~$e$-positive.

\citet{Sta98} posed the further conjecture,
attributed to Gasharov,
that every claw-free graph is Schur-positive;
see also \citet{Gas99}. Very recently,
this conjecture was independently disproved by \citet{Pra26X} and \citet{MM26X},
who identified the same two counterexamples.
The failure of this broad conjecture makes exact results for structured graph
families particularly valuable,
since such families may reveal the coefficient-level mechanisms governing
Schur positivity beyond the exclusion of a single induced subgraph.

In his foundational paper on chromatic symmetric functions,
\citet{Sta95} asked which~$X_G$ are~$e$-positive,
a question that remains open even for trees and spiders.
Spiders are central to this problem because they serve as reduction models for
trees: \citet[Lem.~12]{DSv20} showed that every connected-partition type
admitted by a tree is also admitted by an associated spider,
while the criterion of \citet[Prop.~1.3.3]{Wol97D} turns a missing type into an
obstruction to~$e$-positivity.
Recent progress on degree obstructions for~$e$-positive trees together with the
non-$e$-positivity of all~$4$-legged spiders \citep[Thms.~1.5 and 1.6]{Tom26}
further highlights the importance of understanding the~$e$-positivity of
spiders.

Against this background,
$3$-legged spiders form the natural boundary case for branching.
Their behavior appears to depend strongly on the shortest leg:
\citet[Conj.~5.1]{WW23-DAM} conjectured that the only~$e$-positive spiders with
shortest leg at least~$3$ are $S(8,5,3)$ and $S(14,9,5)$,
and made partial progress toward classifying those with shortest leg~$2$;
see \cref{thm:WW.e+.Sab2}. Subsequently,
\citet[Conj.~2.16]{Tom25-FPSAC} proposed the following complete classification.

\begin{conjecture}[\citeauthor{Tom25-FPSAC}]
\label{conj:Tom.e+.Sabc}
Let $a\ge b\ge c\ge 2$.
Then the spider $S=S(a,b,c)$ is~$e$-positive if and only if
\[
(a,b,c)\in
\brk[c]1{
(5,3,2),\,
(6,4,2),\,
(8,6,2),\,
(9,7,2),\,
(12,4,2),\,
(8,5,3),\,
(14,9,5)
}.
\]
\end{conjecture}

In contrast,
the case in which the shortest leg length is~$1$ behaves differently: infinite
families of both~$e$-positive and non-$e$-positive spiders coexist,
and no comparable classification is currently known;
see \citet{Tom26}.
These observations single out $S(a,b,2)$ as a natural intermediate family.

Our first motivating problem is to determine the Schur positivity of the spiders
$S(a,b,2)$.
The resulting theorem provides an infinite family of Schur-positive spiders
containing claws and gives a uniform result at the~$3$-legged boundary described
above.
The clique-spider results underlying the proof also support a later complete
classification of Schur positivity for three-legged spiders;
see the closing paragraph of \cref{sec:Sab2}.

The proof of \cref{thm:s+.Sab2} uses the path--clique bootstrap in
\cref{lem:path-clique.bootstrap}.
The bootstrap reduces the Schur positivity of a pendant-path extension to that
of shorter extensions and a clique-conjoined base graph.
Thus clique-spiders arise as natural base cases.
For $r,s,t\ge 1$ and $g,h,k\ge 0$,
let~$S^{ghk}_{rst}$ be the graph formed by joining a common center to one vertex
of each of~$K_r$,
$K_s$,
and~$K_t$ by internally disjoint paths of lengths~$g$,
$h$, and~$k$,
respectively.
A path of length zero identifies its endpoints.
Under the convention $r\ge s\ge t$,
the four basic families with $g+h+k\le 1$ are~$S^{000}_{rst}$,
$S^{100}_{rst}$, $S^{010}_{rst}$,
and~$S^{001}_{rst}$; see \cref{fig:basic.clique-spiders}.

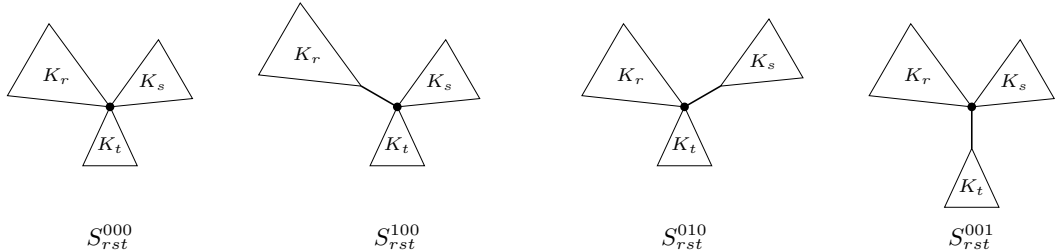
\begin{figure}[ht]
\centering
\begingroup
\newcommand{\schematicbranch}[8]{%
\def\branchoffset{0}%
\ifnum#4=1\relax
\def\branchoffset{.55}%
\fi
\coordinate (#1#2a) at (#3:\branchoffset);
\coordinate (#1#2b) at
($ (#1#2a)+(#3:#6)+({#3+90}:#7) $);
\coordinate (#1#2d) at
($ (#1#2a)+(#3:#6)+({#3-90}:#7) $);
\ifnum#4=1\relax
\draw[leg edge] (c#1)--(#1#2a);
\fi
\path[clique block] (#1#2a)--(#1#2b)--(#1#2d)--cycle;
\node[clique label] at ($ (#1#2a)+(#3:#8) $) {$#5$};
}
\newcommand{\schematicfamily}[6]{%
\begin{scope}[xshift=#1cm]
\coordinate (c#2) at (0,0);
\schematicbranch{#2}{r}{150}{#4}{K_r}{1.25}{.55}{.82}
\schematicbranch{#2}{s}{30}{#5}{K_s}{1.00}{.45}{.65}
\schematicbranch{#2}{t}{-90}{#6}{K_t}{.78}{.36}{.51}
\node[center vertex] at (c#2) {};
\node[family label] at (0,-1.72) {$#3$};
\end{scope}
}
\begin{tikzpicture}[
clique block/.style={draw,line width=.35pt},
leg edge/.style={semithick},
clique label/.style={font=\scriptsize},
center vertex/.style={circle,fill,inner sep=1.15pt},
family label/.style={font=\small}
]
\schematicfamily{1.5}{A}{S^{000}_{rst}}{0}{0}{0}
\schematicfamily{5.3}{B}{S^{100}_{rst}}{1}{0}{0}
\schematicfamily{9.1}{C}{S^{010}_{rst}}{0}{1}{0}
\schematicfamily{12.9}{D}{S^{001}_{rst}}{0}{0}{1}
\end{tikzpicture}
\endgroup
\caption{The four basic $3$-clique-spider families. The filled dot
denotes the common center~$c$. Each triangle represents a clique, with
sizes reflecting $r\ge s\ge t$. A triangle incident with~$c$ represents
a clique containing~$c$; a segment from~$c$ to a disjoint triangle is
the unique leg edge.}
\label{fig:basic.clique-spiders}
\end{figure}

We present our main contributions in two parts.
The first develops two theoretical tools for Schur positivity.
The second gives two kinds of applications.

\noindent\textbf{1. Two theoretical tools.} The first tool is
\cref{lem:alpha-omega.s}.
It bounds the partitions indexing nonzero Schur coefficients of~$X_G$ in terms
of the higher independence numbers,
higher clique numbers, and chromatic number of~$G$.

The second tool consists of \cref{thm:equivalence} for certifying
Schur positivity from a fixed signed~$e_I$-expansion.
For a composition $I=i_1\dots i_k$, write $e_I=e_{i_1}\cdots e_{i_k}$.
An \emph{$e_I$-expansion} of a symmetric function~$F$ is a finite expression
$F=\sum_I c_I e_I$ indexed by compositions~$I$.
It is \emph{signed} if the coefficients~$c_I$ may have either sign,
and \emph{fixed} means that the chosen composition-indexed expression is
retained.
The positive coefficients act as capacities and the negative coefficients as
demands.
An admissible assignment uses only pairs $\alpha\unlhd\beta$,
for which $e_\alpha-e_\beta\spos$.
The matrix, order-ideal,
negative Hall,
and positive Hall conditions in \cref{thm:equivalence} encode this assignment in
four equivalent ways.

\noindent\textbf{2. Two kinds of applications.} Our first application gives
complete~$e$- and Schur-positivity classifications for the four
basic~$3$-clique-spider families.
These classifications show how moving a single leg edge can change positivity,
see~\cref{tab:basic.classification}.

\begin{table}[ht]
\caption{Positivity classification of the four basic
$3$-clique-spider families,
where $r\ge s\ge t\ge 1$.}
\label{tab:basic.classification}
\centering
\renewcommand{\arraystretch}{1.25}
\setlength{\tabcolsep}{4pt}
\begin{tabularx}{\textwidth}{@{}
>{\centering\arraybackslash}p{22mm} >{\raggedright\arraybackslash}p{34mm}
>{\raggedright\arraybackslash}X >{\raggedright\arraybackslash}p{27mm} @{}}
\toprule
\textbf{Family} & \textbf{$e$-positive} & $G\spos$ &
\textbf{Results}
\\
\midrule
$S^{000}_{rst}$ & Exactly when $t=1$ & Exactly when $r\ge st-1$ & $e$:
\cref{thm:barbell}\newline \phantom{$e$: }\cref{thm:e-.000}\newline~$s$:
\cref{thm:s+.000}
\\
\addlinespace[2pt]
$S^{100}_{rst}$,
$S^{010}_{rst}$ & $t=1$,
or $t=2$ and $r=s$ & Always & $e$: \cref{thm:e+.**0}\newline~$s$:
\cref{thm:s+.100}\newline \phantom{$s$: }\cref{cor:s+.010}
\\
\addlinespace[2pt]
$S^{001}_{rst}$ & $s=t=1$,
or $(r,s,t)=(2,2,2)$ &
If and only if at least one of the following holds:
\newline (i) $s=t$;
\newline (ii) $r\ge 2s-1$;
\newline (iii) $3t>r+s$;
\newline (iv) $(r,s)=(2t-1,\,t+1)$;
\newline (v) $(r,s,t)=(6,6,4)$. & $e$:
\cref{cor:e+.S001}\newline~$s$: \cref{cor:s+.010}\newline
\phantom{$s$: }\cref{thm:s.001}
\\
\bottomrule
\end{tabularx}
\end{table}

The second application,
with the aid of selected Schur positivity results in
\cref{tab:basic.classification},
proves the Schur positivity of all~$3$-legged spiders with the shortest leg
length~$2$;
see \cref{thm:s+.Sab2}.
The proof also needs a path--clique bootstrap in
\cref{lem:path-clique.bootstrap}.
As a separate~$e$-positivity contribution,
\cref{thm:e+.Sab2.nonzero.residue} verifies \cref{conj:Tom.e+.Sabc} for all
spiders of the form $S(a,b,2)$ with $3\nmid b$.
Within the family $S(a,b,2)$,
only the cases $3\mid b$ and $b\ge 12$ of \cref{conj:Tom.e+.Sabc}
remain unresolved.

The remainder of the paper is organized as follows.
Preliminaries are collected in \cref{sec:pre}.
\Cref{sec:schur.tools} develops \cref{lem:alpha-omega.s} and
\cref{thm:equivalence}.
\Cref{sec:000} gives the positivity classification of the
graphs~$S^{000}_{rst}$.
\Cref{sec:S100:S010} classifies the families~$S^{100}_{rst}$ and~$S^{010}_{rst}$.
The last family~$S^{001}_{rst}$ is classified in \cref{sec:S001}.
Finally, \cref{sec:Sab2} treats spiders $S(a,b,2)$,
including both their Schur positivity and the nonzero-residue cases of
\cref{conj:Tom.e+.Sabc}.

\section{Preliminaries}
\label{sec:pre}

Let $\mathbb Q$ be the field of rational numbers,
and $\mathbb Z$ the ring of integers.
Let $\mathbb Z_{\ge 0}=\{0,1,2,\dots\}$ and $\mathbb Z_{\ge 1}=\{1,2,\dots\}$.
For $n\in\mathbb Z_{\ge 1}$,
write $[n]=\{1,2,\dots,n\}$. For any proposition~$P$,
let $\mathbf{1}(P)$ be its \emph{indicator},
equal to~$1$ when~$P$ is true and to~$0$ otherwise.

A \emph{weak composition} of $n\in\mathbb Z_{\ge 0}$ is a finite or infinite
sequence of nonnegative integers with finite support and sum~$n$,
written $I\vDash_w n$.
We identify sequences differing only by trailing zeros and use the unique finite
representative that is empty or ends in a positive part;
$\ell(I)$ is its number of parts.
The lowercase form of the symbol denotes the parts,
so $I=i_1\dots i_{\ell(I)}$.
Fixed-length notation retains trailing zeros: $abc\vDash_w n$ means
$(a,b,c)\in\mathbb Z_{\ge 0}^3$ with $a+b+c=n$.
The \emph{size} of~$I$ is $\abs{I}=\sum_{j=1}^{\ell(I)}i_j$.
For $n\in\mathbb Z_{\ge 1}$,
a \emph{composition} $I\vDash n$ is a weak composition of size~$n$ with positive
parts.
The unique weak composition of~$0$ is the \emph{empty composition}.

A \emph{partition}~$\lambda$ is a weakly decreasing composition,
denoted $\lambda\vdash n$.
Repeated parts in~$\lambda$ are written exponentially.
The \emph{Young diagram} of~$\lambda$,
drawn in the French convention,
is the left-justified array of cells $(i,j)\in\mathbb Z_{\ge 1}^2$ such that
$1\le i\le \ell(\lambda)$ and $1\le j\le \lambda_i$,
where row indices increase from bottom to top and column indices from left to
right,
each starting from~$1$.
The \emph{conjugate}~$\lambda'$ is the partition whose~$j$th part is the number
of cells in column~$j$;
equivalently,
\[
\lambda'_j
=\abs{\{i\in[\ell(\lambda)]\colon \lambda_i\ge j\}}
\qquad(1\le j\le \lambda_1).
\]
In particular, $\lambda'_1=\ell(\lambda)$.
For any weak composition~$I$,
denote by~$\rho(I)$ the underlying partition obtained by deleting its zero parts
and rearranging the remaining parts in weakly decreasing order.

We work in the algebra~$\mathrm{Sym}$ of symmetric functions over $\mathbb Q$
and use its monomial,
elementary, complete homogeneous,
power-sum, and Schur bases $m_\lambda,e_\lambda,h_\lambda,p_\lambda,s_\lambda$;
see, for example,
\citet{Mac95B}. For any weak composition~$I$,
set
\[
e_I
=e_{\rho(I)}
=\prod_j e_{i_j},
\qquad
h_I
=h_{\rho(I)}
=\prod_j h_{i_j},
\]
where $e_0=h_0=1$.
Every basis element with a negative subscript is interpreted as zero.
For any of these bases,
write $\coeff{b_\lambda}{F}$ for the coefficient of~$b_\lambda$ in a symmetric
function~$F$.
We say that~$F$ is \emph{$b$-positive} if $\coeff{b_\lambda}{F}\ge 0$ for any
partition~$\lambda$.

For partitions~$\lambda$ and~$\mu$ such that $\lambda_i\le \mu_i$ for every~$i$,
where missing parts are interpreted as zero,
the \emph{skew diagram} $\mu/\lambda$ consists of the cells of the Young diagram
of~$\mu$ that do not belong to the Young diagram of~$\lambda$.
Its \emph{size} is $\abs{\mu}-\abs{\lambda}$.
A skew diagram is called a \emph{horizontal strip} if it has no two cells in the
same column.
It is called a \emph{vertical strip} if it has no two cells in the same row.
The \emph{Hall involution} is the automorphism
$\omega\colon\mathrm{Sym}\to\mathrm{Sym}$ defined equivalently by any of the
following three equations:
\[
\omega(e_m)
=h_m,
\qquad
\omega(s_\lambda)
=s_{\lambda'},
\qquad
\omega(p_\lambda)
=\varepsilon_\lambda p_\lambda,
\]
where $\varepsilon_\lambda=(-1)^{\abs{\lambda}-\ell(\lambda)}$ is the
\emph{sign} of~$\lambda$.
In particular, $\varepsilon_d=(-1)^{d-1}$ for any $d\in\mathbb Z_{\ge 1}$.

\citet[Eqs.~(5.16)--(5.17), p.~73]{Mac95B} gives the \emph{Pieri rules}.

\begin{proposition}[The Pieri rules]
\creflabel[pieri]{prop:Pieri}
For every partition~$\lambda$ and integer $d\ge 0$,
\[
e_d s_\lambda
=\sum_\mu s_\mu
\quad\text{and}\quad
h_d s_\lambda
=\sum_\mu s_\mu,
\]
where the first sum is over the partitions~$\mu$ for which $\mu/\lambda$ is a
vertical strip of size~$d$,
and the second sum is over the partitions~$\mu$ for which $\mu/\lambda$ is a
horizontal strip of size~$d$.
\end{proposition}

All graphs~$G$ in this paper are finite and simple.
We write $V(G)$ for the vertex set of~$G$ and call~$\abs{V(G)}$ the \emph{order}
of~$G$.
A \emph{proper coloring} assigns positive integers to its vertices so that
adjacent vertices receive different integers.
\citet{Sta95} defined the \emph{chromatic symmetric function} of~$G$ by
\[
X_G
=\sum_{\kappa\ \mathrm{proper}}\prod_{v\in V(G)}x_{\kappa(v)}.
\]
For example,
for the complete graph~$K_m$ on~$m$ vertices,
$X_{K_m}=m!e_m$.
A graph is called \emph{$e$-positive} or \emph{Schur positive} when its
chromatic symmetric function has that property.

A \emph{rooted graph} $(G,u)$ is a graph with a distinguished vertex~$u$.
For any rooted graphs $(G,u)$ and $(H,v)$,
and any $l\in\mathbb Z_{\ge 0}$,
let $P^l(G,H)$ be the graph obtained by joining~$u$ and~$v$ with a path of
length~$l$.
We abbreviate $G^l=P^l(G,K_1)$. In particular,
$G^0=G$.
One may omit the roots for any complete graph,
since their vertices are equivalent under automorphisms.

For any composition $I=i_1\cdots i_{\ell(I)}$,
the \emph{spider} $S(I)$ is obtained by identifying one endpoint from each of
the paths of lengths $i_1,\dots,i_{\ell(I)}$.
The common vertex is its \emph{center},
and the paths are its \emph{legs}.
The graphs $P^k(K_a,K_b)$ are called \emph{barbells}.
For a composition $I=i_1\cdots i_{\ell(I)}$,
set
\[
w_I
=i_1\prod_{j=2}^{\ell(I)}(i_j-1).
\]
We use barbell~$e$-positivity from \citet[Thm.~3.7]{WZ25}
and the path formula of \citet{SW16}.

\begin{theorem}[Barbells and paths]
\label{thm:barbell}
For any integers $a,b\ge 1$ and $k\ge 0$,
the barbell $P^k(K_a,K_b)$ is~$e$-positive.
Moreover, for any $n\ge 1$, $X_{P_n}
=\sum_{I\vDash n}w_I e_I$.
\end{theorem}

Define the \emph{normalization} of~$X_G$ for $G=S_{rst}^{ghk}$ by
\[
Y_G
=\frac{X_G}{(r-1)!(s-1)!(t-1)!}.
\]
Since the normalizing factor is positive,
the chromatic symmetric function~$X_G$ and its normalization~$Y_G$ have the
same~$e$-positivity and Schur-positivity properties.

\section{General tools for Schur positivity}
\label{sec:schur.tools}

We develop two tools for the positivity arguments that follow.
\Cref{lem:alpha-omega.s} bounds the partitions in the Schur support of a
chromatic symmetric function.
\Cref{thm:equivalence} certifies Schur positivity of a symmetric function from a
fixed signed~$e_I$-expansion.
We use the support bounds in \cref{sec:000} and apply \cref{thm:equivalence}
through the prefix Hall criterion (\cref{cor:s+.Hall.prefix}) in
\cref{sec:S100:S010}.

\subsection{A Schur alpha--omega lemma}
\label{sec:alpha-omega}

\citet[Lems.~2.2 and~2.10, and Rem.~2.3]{ST26} introduced the alpha--omega lemma
for the elementary expansion of a chromatic symmetric function and extended it
from the independence and clique numbers to their higher analogues.
In this subsection,
we prove the corresponding result for Schur coefficients in
\cref{lem:alpha-omega.s}.

For a graph~$G$,
a \emph{stable set} is a vertex set spanning no edge,
while a \emph{clique} is a vertex set inducing a complete graph.
The \emph{independence number} $\alpha(G)$ and the \emph{clique number}
$\omega(G)$ are the maximum sizes of a stable set and a clique,
respectively. More generally,
for $k\ge 1$,
let $\alpha_k(G)$ (resp.~$\omega_k(G)$) be the maximum total number of vertices
in~$k$ pairwise vertex-disjoint stable sets (resp.~cliques).
In particular, $\alpha_1(G)=\alpha(G)$ and $\omega_1(G)=\omega(G)$.
The \emph{chromatic number} $\chi(G)$ is the minimum number of colors in a
proper coloring of~$G$.

We begin by recalling the higher-parameter version of the elementary
alpha--omega lemma.

\begin{lemma}[{Elementary alpha--omega lemma}, \citeauthor{ST26}]
\label{lem:alpha-omega.e}
Let~$G$ be a graph and~$\lambda$ a partition such that
$\coeff{e_\lambda}{X_G}\ne 0$.
Then, for any $k\ge 1$,
\[
\lambda'_1+\cdots+\lambda'_k
\le \alpha_k(G)
\qquad\text{and}\qquad
\lambda_1+\cdots+\lambda_k
\ge \omega_k(G),
\]
where missing parts of~$\lambda$ and~$\lambda'$ are interpreted as zero.
Moreover, $\lambda_1\ge \chi(G)$.
\end{lemma}

We now establish \cref{lem:alpha-omega.s}.
A partition $\lambda\vdash n$ is said to be \emph{dominated} by a partition
$\mu\vdash n$,
denoted $\lambda\unlhd\mu$, if
\[
\sum_{i=1}^k\lambda_i
\le \sum_{i=1}^k\mu_i
\quad\text{for any $k\ge 1$},
\]
where missing parts are zeros.
We use~$\lhd$ for \emph{strict dominance} and~$\unrhd$ and~$\rhd$ for the
reverse and strict reverse relations,
respectively. For weak compositions~$I$ and~$J$,
dominance means that for the underlying partitions:
\[
I\unlhd J\quad\Longleftrightarrow\quad\rho(I)\unlhd\rho(J).
\]

\begin{lemma}[Schur alpha--omega lemma]
\creflabel[schuralphaomega]{lem:alpha-omega.s}
Let~$G$ be a graph and~$\lambda$ a partition such that
$\coeff{s_\lambda}{X_G}\ne 0$.
Then, for any $k\ge 1$,
\[
\lambda_1+\cdots+\lambda_k
\le \alpha_k(G)
\qquad\text{and}\qquad
\lambda_1'+\dots+\lambda_k'
\ge \omega_k(G),
\]
where missing parts of~$\lambda$ and~$\lambda'$ are interpreted as zero.
Moreover, $\ell(\lambda)\ge \chi(G)$.
\end{lemma}

\begin{proof}
Write $X_G=\sum_{\nu\vdash\abs{V(G)}}c_\nu e_\nu$.
Since $\coeff{s_\lambda}{X_G}\ne 0$,
there exists a partition~$\nu$ such that $c_\nu\ne 0$ and
$\coeff{s_\lambda}{e_\nu}=K_{\lambda',\nu}\ne 0$,
where the~$K_{\mu\nu}$ are Kostka numbers.
Their dominance support property gives $\lambda'\unrhd\nu$,
and hence $\lambda\unlhd\nu'$;
see \citet[Eqs.~(6.4) and~(6.6)(vi), pp.~101--102]{Mac95B}.
Applying \cref{lem:alpha-omega.e} to~$\nu$ yields,
for every $k\ge 1$,
\[
\sum_{i=1}^k\lambda_i
\le \sum_{i=1}^k\nu'_i
\le \alpha_k(G)
\quad\text{and}\quad
\sum_{i=1}^k\lambda'_i
\ge \sum_{i=1}^k\nu_i
\ge \omega_k(G).
\]
Moreover, $\ell(\lambda)=\lambda'_1\ge\nu_1\ge\chi(G)$.
\end{proof}

In comparison with \cref{lem:alpha-omega.e},
the roles of~$\lambda$ and~$\lambda'$ in \cref{lem:alpha-omega.s} are
interchanged.
\Cref{lem:alpha-omega.s} will be used in \cref{sec:000}.
The final result $\chi(G)\le \ell(\lambda)$ is strictly stronger than
$\omega(G)\le \ell(\lambda)$ if and only if $\chi(G)>\omega(G)$.

\begin{example}
The \emph{wheel}~$W_n$ is the graph on~$n$ vertices obtained from the
cycle~$C_{n-1}$ and a vertex~$v$ by linking~$v$ to all the other $n-1$ vertices.
One may compute by Maple and obtain
\[
X_{W_6}
=20s_{2^21^2}+60s_{21^4}+240s_{1^6}.
\]
Here for every partition~$\lambda$ with $\coeff{s_\lambda}{X_{W_6}}\ne 0$,
we have
\[
\lambda_1
\le \alpha(W_6)
=2,\quad
\ell(\lambda)
\ge \omega(W_6)
=3,
\quad\text{and}\quad
\ell(\lambda)
\ge \chi(W_6)
=4.
\]
As a consequence,
the stronger result $\ell(\lambda)\ge \chi(W_6)$ implies that
$\coeff{s_{2^3}}{X_{W_6}}=0$.
\end{example}

\subsection{Equivalent dominance-matching criteria for Schur positivity}
\label{sec:s+.criteria}

We establish four equivalent sufficient criteria for proving the
Schur positivity of a symmetric function from a fixed signed~$e_I$-expansion.
We call them the matrix,
order-ideal, negative Hall,
and positive Hall criteria, respectively.
We first recall a standard description of dominance covers due to
\citet{Mui1902}.

\begin{lemma}[\citeauthor{Mui1902}]
\label{lem:transfer-dominance}
If $\mu\lhd\lambda$ are partitions of the same integer,
then there is a sequence
\[
\mu
=\eta^{(0)}\lhd\eta^{(1)}\lhd\cdots
\lhd\eta^{(N)}
=\lambda
\]
in which each~$\eta^{(t+1)}$ is obtained from~$\eta^{(t)}$ by replacing two
parts~$a$ and~$b$,
where $a\ge b\ge 1$, with $a+1$ and $b-1$,
then deleting a zero part if necessary and reordering.
\end{lemma}

The following folkloric sufficient condition underlies all four certificates,
and we include a proof for completeness.

\begin{proposition}
\label{prop:s+.e-e}
Let $\lambda,\mu\vdash n$. If $\mu\unlhd\lambda$,
then $e_\mu-e_\lambda\spos$.
\end{proposition}

\begin{proof}
If $\lambda=\mu$, the assertion is immediate.
Otherwise,
by the transfer characterization in \cref{lem:transfer-dominance},
it is enough to consider the case in which~$\lambda$ is obtained from~$\mu$ by
replacing two parts~$a$ and~$b$,
where $a\ge b\ge 1$,
by $a+1$ and $b-1$ and then deleting a zero part if necessary.
Let~$\gamma$ consist of the other parts.
Since $e_r=s_{1^r}$, \cref{prop:Pieri} gives
\begin{align*}
e_\mu-e_\lambda
=(e_a e_b-e_{a+1}e_{b-1})e_\gamma
=\brk4{
\sum_{j=0}^{b}s_{2^j1^{a+b-2j}}
-\sum_{j=0}^{b-1}s_{2^j1^{a+b-2j}}}
e_\gamma
=s_{2^b1^{a-b}}e_\gamma\spos.
\qedhere
\end{align*}
\end{proof}

Throughout this subsection, let
\begin{equation}
\creflabel[def]{def:f.criterion}
f
=\sum_{\alpha\in\mathcal A}a_\alpha e_\alpha
-\sum_{\beta\in\mathcal B}b_\beta e_\beta,
\end{equation}
where $\mathcal A$ and $\mathcal B$ are finite sets of compositions of the same
integer~$n$,
and $a_\alpha,b_\beta\in\mathbb Z_{\ge 0}$.
Rational coefficients reduce to this formulation by clearing denominators.
For $\beta\in\mathcal B$ and $B\subseteq\mathcal B$, set
\[
\Gamma(\beta)
=\{\alpha\in\mathcal A\colon\alpha\unlhd\beta\},
\qquad
\Gamma(B)
=\bigcup_{\beta\in B}\Gamma(\beta).
\]
A set~$\mathcal I$ of compositions is an \emph{order ideal} if
$\alpha\in\mathcal I$ and $\beta\unlhd\alpha$ imply $\beta\in\mathcal I$.
We call the Hall conditions negative and positive because they are indexed by
subsets of~$\mathcal B$ and~$\mathcal A$,
respectively.

\begin{theorem}[Dominance-matching criteria]
\creflabel[dominancematching]{thm:equivalence}
For the signed expansion in \cref{def:f.criterion},
the following conditions are equivalent:
\begin{enumerate}[label=\textup{(\roman*)}]
\item
\creflabel[itm]{itm:matching.matrix}
\emph{Matrix condition.} There are nonnegative integers~$c_{\alpha\beta}$,
indexed by $(\alpha,\beta)\in\mathcal A\times\mathcal B$, such that
\[
c_{\alpha\beta}
=0\quad\text{unless }\alpha\unlhd\beta,
\qquad
\sum_\alpha c_{\alpha\beta}
=b_\beta,
\qquad
\sum_\beta c_{\alpha\beta}
\le a_\alpha.
\]
\item
\creflabel[itm]{itm:matching.ideal}
\emph{Order-ideal condition.}
$\Delta(\mathcal I)
:=\sum_{\alpha\in\mathcal A\cap\mathcal I}a_\alpha
-\sum_{\beta\in\mathcal B\cap\mathcal I}b_\beta\ge 0$
for any dominance order ideal~$\mathcal I$.
\item
\creflabel[itm]{itm:matching.Hall}
\emph{Negative Hall condition.}
$\sum_{\alpha\in\Gamma(B)}a_\alpha\ge \sum_{\beta\in B}b_\beta$ for any
$B\subseteq\mathcal B$.
\item
\creflabel[itm]{itm:matching.Hall.positive}
\emph{Positive Hall condition.}
$\sum_{\alpha\in A}a_\alpha\ge \sum_{\Gamma(\beta)\subseteq A}b_\beta$ for any
$A\subseteq\mathcal A$.
\end{enumerate}
If any of these conditions holds,
then~$f\spos$.
\end{theorem}

\begin{proof}
We prove the equivalences cyclically.

\noindent\textup{(i)}~$\Rightarrow$ \textup{(ii)}. If $\beta\in\mathcal I$ and
$c_{\alpha\beta}>0$,
then $\alpha\in\mathcal I$. Hence,
by \cref{itm:matching.matrix},
\[
\sum_{\alpha\in\mathcal A\cap\mathcal I}a_\alpha
\ge \sum_{\alpha\in\mathcal A\cap\mathcal I}
\sum_{\beta\in\mathcal B\cap\mathcal I}c_{\alpha\beta}
=\sum_{\beta\in\mathcal B\cap\mathcal I}b_\beta.
\]

\noindent\textup{(ii)}~$\Rightarrow$ \textup{(iii)}.
Let $\mathcal I$ be the ideal generated by~$B$.
Then, by \cref{itm:matching.ideal},
\[
\mathcal A\cap\mathcal I
=\Gamma(B),
\qquad
B\subseteq\mathcal B\cap\mathcal I,
\qquad\text{and}\qquad
\sum_{\alpha\in\Gamma(B)}a_\alpha
\ge \sum_{\beta\in\mathcal B\cap\mathcal I}b_\beta
\ge \sum_{\beta\in B}b_\beta.
\]

\noindent\textup{(iii)}~$\Rightarrow$ \textup{(iv)}. For $A\subseteq\mathcal A$,
put $B_A=\{\beta\in\mathcal B\colon\Gamma(\beta)\subseteq A\}$. Since
$\Gamma(B_A)\subseteq A$,
\cref{itm:matching.Hall} gives
\[
\sum_{\alpha\in A}a_\alpha
\ge \sum_{\alpha\in\Gamma(B_A)}a_\alpha
\ge \sum_{\beta\in B_A}b_\beta.
\]

\noindent\textup{(iv)}~$\Rightarrow$ \textup{(i)}.
Replace each demand~$\beta$ by~$b_\beta$ vertices and each capacity~$\alpha$
by~$a_\alpha$ vertices.
Join all copies of~$\beta$ to all copies of~$\alpha$ when $\alpha\unlhd\beta$.
If a set~$U$ of demand vertices uses the types~$B$,
then $B\subseteq B_{\Gamma(B)}$.
\Cref{itm:matching.Hall.positive} applied
to~$\Gamma(B)$ gives
\[
\abs{U}
\le \sum_{\beta\in B}b_\beta
\le \sum_{\beta\in B_{\Gamma(B)}}b_\beta
\le \sum_{\alpha\in\Gamma(B)}a_\alpha
=\abs{N(U)}.
\]
Hall's marriage theorem gives a matching saturating all demand vertices;
see \citet[Thm.~2.1.2]{Die25B}.
Counting matched pairs of each type gives the required matrix.

Finally, under \cref{itm:matching.matrix},
\[
f
=\sum_{\alpha,\beta}c_{\alpha\beta}(e_\alpha-e_\beta)
+\sum_\alpha\brk3{a_\alpha-\sum_\beta c_{\alpha\beta}}e_\alpha.
\]
Thus $f\spos$ by \cref{prop:s+.e-e} and the inequalities
$\sum_\beta c_{\alpha\beta}\le a_\alpha$.
\end{proof}

The prefix criterion used in our applications follows immediately.

\begin{corollary}[Prefix Hall criterion]
\label{cor:s+.Hall.prefix}
Let~$f$ be defined by \cref{def:f.criterion}.
Suppose that the compositions in~$\mathcal A$ can be ordered as
$\alpha^{(0)},\dots,\alpha^{(m)}$,
so that for each $\beta\in\mathcal B$ there is an integer $0\le M_\beta\le m$
with
\[
\Gamma(\beta)
=\{\alpha^{(i)}\colon 0
\le i
\le M_\beta\}.
\]
Then $f\spos$ if
$\sum_{i=0}^{q}a_{\alpha^{(i)}}\ge \sum_{M_\beta\le q}b_\beta$ for every
$0\le q\le m$.
\end{corollary}
\begin{proof}
For $A\subseteq\mathcal A$,
let~$q$ be the largest index for which
$\{\alpha^{(0)},\dots,\alpha^{(q)}\}\subseteq A$.
Set $q=-1$ if no such index exists.
Then
\[
\sum_{\Gamma(\beta)\subseteq A}b_\beta
=\sum_{M_\beta\le q}b_\beta
\le \sum_{i=0}^{q}a_{\alpha^{(i)}}
\le \sum_{\alpha\in A}a_\alpha.
\]
This verifies \cref{itm:matching.Hall.positive} of
\cref{thm:equivalence},
so $f\spos$.
\end{proof}

In \cref{sec:S100:S010},
we apply \cref{cor:s+.Hall.prefix} to prove Schur positivity for the
families~$S^{100}_{rst}$ and~$S^{010}_{rst}$.

Suppose that $f=\sum_{\nu\vdash n}c_\nu e_\nu$. Then
$\Delta(\mathcal I)=\sum_{\nu\in\rho(\mathcal I)}c_\nu$.
Thus certificate feasibility depends only on the elementary coefficients of~$f$.
Composition indices can still be retained to construct the matrices and the
prefix orderings used in the applications.

We remark that the order-ideal criterion cannot be replaced by a criterion
involving only principal order ideals.
For instance, consider the symmetric function
\[
f
=e_\alpha-e_\beta-e_\gamma+e_\delta,
\quad\text{where $(\alpha,\beta,\gamma,\delta)=\brk1{4^22,\,62^2,\,541,\,631}$.}
\]
We have $\alpha\unlhd\beta$, $\alpha\unlhd\gamma$,
and $\beta\vee\gamma=\delta$. Thus,
for every principal ideal $\downarrow\!\lambda$,
we have
\[
\Delta(\downarrow\!\lambda)
=\mathbf{1}(\alpha\unlhd\lambda)
+\mathbf{1}(\delta\unlhd\lambda)
-\mathbf{1}(\beta\unlhd\lambda)
-\mathbf{1}(\gamma\unlhd\lambda)
\ge 0.
\]
Indeed,
if exactly one of~$\beta$ and~$\gamma$ lies below~$\lambda$,
then~$\alpha$ lies below~$\lambda$ as well;
and if both lie below~$\lambda$,
then their least upper bound~$\delta$ lies below~$\lambda$.
On the other hand, for the non-principal ideal
\[
\mathcal I
=(\downarrow\!\beta)\cup(\downarrow\!\gamma),
\]
we have $\Delta(\mathcal I)=1-1-1=-1$.
Hence positivity on principal ideals does not imply positivity on all order
ideals.

\section{\texorpdfstring{Positivity characterization for clique-spiders
$S_{rst}^{000}$}{Positivity characterization for clique-spiders S000}}
\label{sec:000}

In this section,
we determine the~$e$- and Schur positivity of the graphs~$S_{rst}^{000}$.
When $t=1$,
this graph is the barbell $P^0(K_s,K_r)$ and is~$e$-positive by
\cref{thm:barbell}.
Hence, throughout this section,
we can suppose that
\[
r
\ge s
\ge t
\ge 2,
\quad
G
=S_{rst}^{000},
\quad\text{and}\quad
n
=r+s+t-2.
\]
Then the graph~$G$ has~$n$ vertices.
We begin by expressing the symmetric function $\omega Y_G$ in a form that will
be useful for the~$e$-coefficient calculation.

\begin{theorem}
\label{thm:omegaY.000}
Let $r\ge s\ge t\ge 2$, $G=S_{rst}^{000}$ and $n=r+s+t-2$.
Then
\[
\omega Y_G
=\sum_{i=0}^{r-1}
\sum_{j=0}^{s-1}
\sum_{k=0}^{t-1}
p_{n-i-j-k}
h_{ijk}.
\]
\end{theorem}

\begin{proof}
We claim that,
for any $m\ge 0$ and $q\ge 1$,
if~$e_m^{(q)}$ denotes the elementary symmetric function~$e_m$ in all variables
$x_1,x_2,\dots$ except~$x_q$,
then
\[
e_m^{(q)}
=\sum_{j=0}^m(-x_q)^{m-j}e_j.
\]
In fact, using generating functions,
we obtain
\[
\sum_{m\ge 0}e_m^{(q)}z^m
=\prod_{\ell\ne q}(1+x_\ell z)
=\frac{1}{1+x_qz}
\brk4{\prod_{\ell\ge 1}(1+x_\ell z)}
=\brk4{\sum_{j\ge 0}(-x_qz)^j}
\brk4{\sum_{m\ge 0}e_m z^m}.
\]
Extracting the coefficient of~$z^m$ and replacing~$j$ by $m-j$ gives the claimed
identity.

We now compute~$Y_G$.
Fix the color~$q$ of the center~$c$.
The $r-1$ vertices in $K_r-c$ receive distinct colors,
none equal to~$q$;
their contribution to~$X_G$ is $(r-1)!e_{r-1}^{(q)}$.
Similarly,
the subgraphs $K_s-c$ and $K_t-c$ contribute $(s-1)!e_{s-1}^{(q)}$ and
$(t-1)!e_{t-1}^{(q)}$ to~$X_G$,
respectively. Therefore,
\[
Y_G
=\sum_{q\ge 1}
x_q e_{r-1}^{(q)}e_{s-1}^{(q)}e_{t-1}^{(q)}.
\]
The claim gives, for every $m\in\{r,s,t\}$,
\[
e_{m-1}^{(q)}
=\sum_{i=0}^{m-1}
\varepsilon_{m-i}
x_q^{m-1-i}e_i.
\]
Substituting these three identities into the formula for~$Y_G$,
we obtain
\[
Y_G
=\sum_{i=0}^{r-1}
\sum_{j=0}^{s-1}
\sum_{k=0}^{t-1}
\varepsilon_{n-i-j-k}
\sum_{q\ge 1}x_q^{n-i-j-k}
e_{ijk}
=\sum_{i=0}^{r-1}
\sum_{j=0}^{s-1}
\sum_{k=0}^{t-1}
\varepsilon_{n-i-j-k}
p_{n-i-j-k}
e_{ijk}.
\]
Applying the Hall involution~$\omega$ to both sides,
we obtain the desired formula.
\end{proof}

Now we can show that~$G$ is never~$e$-positive.

\begin{theorem}
\label{thm:e-.000}
Let $r\ge s\ge t\ge 2$, $G=S_{rst}^{000}$ and $n=r+s+t-2$.
Then
\[
\coeff{e_{r(s+t-2)}}{Y_G}
=\begin{dcases*}
-n,& if $r<s+t-2$,
\\
-r,& if $r=s+t-2$,
\\
r-n,& if $r>s+t-2$.
\end{dcases*}
\]
As a consequence, $G$ is not~$e$-positive.
\end{theorem}

\begin{proof}
Note that $\coeff{e_\lambda}{Y_G}=\coeff{h_\lambda}{\omega Y_G}$.
The Girard--Waring power-sum formula in
\citet[Chap.~I, \S2, Ex.~20, p.~33]{Mac95B} gives
\begin{equation}
\label{eq:power-sum-two-h-parts}
p_n
=\varepsilon_n
\sum_{\lambda\vdash n}
\varepsilon_\lambda
\frac{n\brk1{\ell(\lambda)-1}!}
{\prod_{i\ge 1}\mathsf v_i(\lambda)!}\,h_\lambda
=\sum_{\lambda\vdash n}
(-1)^{\ell(\lambda)-1}
\frac{n\brk1{\ell(\lambda)-1}!}
{\prod_{i\ge 1}\mathsf v_i(\lambda)!}\,h_\lambda,
\end{equation}
where $\mathsf v_i(\lambda)$ is the multiplicity of the part~$i$ in~$\lambda$.
Let $L=s+t-2$.
We extract~$h_{rL}$ from the formula in \cref{thm:omegaY.000}.
Since
\[
j
\le s-1<r,
\qquad
k
\le t-1<r,
\qquad\text{and}\qquad
j,k<L,
\]
neither~$h_j$ nor~$h_k$ can supply a positive part of~$h_{rL}$.
Hence every contributing term has $j=k=0$.
Moreover, $i\le r-1$,
so~$h_i$ cannot supply the part~$r$;
it can supply the part~$L$ exactly when $L<r$.
If $r\le L$, then only $(i,j,k)=(0,0,0)$ contributes,
and the coefficient $\coeff{h_{rL}}{\omega Y_G}$ is directly obtained by
\cref{eq:power-sum-two-h-parts}.
If $r>L$,
then the term with $(i,j,k)=(0,0,0)$ contributes $-n$,
while the term with $(i,j,k)=(L,0,0)$ is $p_r h_L$ and contributes~$r$.
Therefore $\coeff{h_{rL}}{\omega Y_G}=-n+r$.

Since all three values are negative,
$G$ is not~$e$-positive.
\end{proof}

We now turn to the Schur-positivity characterization of~$G$.
A \emph{stable partition} of a graph~$H$ is a partition of $V(H)$ into stable
sets.
Its \emph{type} is the partition of the number~$\abs{V(H)}$ of vertices formed
by the block sizes.
\citet{Sta98} defined a graph~$H$ to be \emph{nice} if,
for every pair of partitions $\mu\unlhd\lambda$ of~$\abs{V(H)}$,
the existence of a stable partition of type~$\lambda$ implies the existence of a
stable partition of type~$\mu$.
\citet[Prop.~1.5]{Sta98} proved the following necessary condition for
Schur positivity.

\begin{proposition}[\citeauthor{Sta98}]
\label{prop:s+.nice}
Every Schur-positive graph is nice.
\end{proposition}

We first show that~$G$ fails this necessary condition when $r\le s+t-2$.

\begin{lemma}
\label{lem:nice-.000}
Let $r\ge s\ge t\ge 2$. If $r\le s+t-2$,
then the graph~$S_{rst}^{000}$ is not nice.
\end{lemma}

\begin{proof}
The graph~$G$ has a stable partition of type $\lambda=3^{t-1}2^{s-t}1^{r-s+1}$,
and it is routine to verify that $\lambda\rhd\mu$ for
\[
\mu
=\begin{dcases*}
2^{n/2},& if $n$ is even,
\\
32^{(n-3)/2},& if $n$ is odd.
\end{dcases*}
\]
However,
$G$ has no stable partition of type~$\mu$: the block containing the center is
necessarily a singleton,
whereas~$\mu$ has no part equal to~$1$.
Thus~$G$ is not nice.
\end{proof}

We next derive identities by fixing the color of the center.
They will give the Schur expansion of~$S^{000}_{rst}$ and the~$e$-expansions of
the single-edge families.
Write $Z_{RST}=Y_{S^{000}_{RST}}$.

\begin{lemma}
\label{lem:clique-spider.center}
Let $R,S,T\ge 1$ and $N=R+S+T-2$. Then
\begin{align}
\label{eq:ZRST.e}
Z_{RST}
&=\sum_{abc\vDash_w N,\ b<S,\ c<T}
(a-b)e_{abc}
-\sum_{abc\vDash_w N,\
a\ge R,\ b\ge S}
(b-c)e_{abc},
\quad\text{and}
\\
\label{eq:S100.ZB}
Y_{S^{100}_{RST}}
&=Z_{(R+1)ST}+(R-1)e_R Z_{ST1}.
\end{align}
\end{lemma}

\begin{proof}
Recall that~$e_j^{(q)}$ denotes the elementary symmetric function~$e_j$ in all
variables $x_1,x_2,\dots$ except~$x_q$.
Fixing the color of the center gives
\[
Z_{RST}
=\sum_q x_q e_{R-1}^{(q)}e_{S-1}^{(q)}e_{T-1}^{(q)}.
\]
For $S\ge 2$, using $\sum_q x_q e_{j-1}^{(q)}=je_j$ and
$e_j=e_j^{(q)}+x_q e_{j-1}^{(q)}$,
we can deduce that
\begin{align*}
Z_{RS1}
&=e_{S-1}\sum_q x_q e_{R-1}^{(q)}
-\sum_q x_q^2e_{R-1}^{(q)}e_{S-2}^{(q)}
=Re_{R(S-1)}
-\sum_q x_q\brk1{e_R-e_R^{(q)}}e_{S-2}^{(q)}
\\
&=(R-S+1)e_{R(S-1)}+Z_{(R+1)(S-1)1}.
\end{align*}
Thus iterating until the second clique has size~$1$ gives
\begin{align*}
Z_{RS1}
&=\sum_{k=0}^{S-2}(R-S+1+2k)e_{(R+k)(S-1-k)}
+(R+S-1)e_{R+S-1}
\\
&=\sum_{i=0}^{S-1}(R+S-1-2i)e_{(R+S-1-i)i}.
\end{align*}
Let $T\ge 2$. Then
\begin{align}
\notag
Z_{RST}
&=e_{T-1}\sum_q x_q e_{R-1}^{(q)}e_{S-1}^{(q)}
-\sum_q x_q^2e_{R-1}^{(q)}e_{S-1}^{(q)}e_{T-2}^{(q)}
=e_{T-1}Z_{RS1}
-\sum_q x_q\brk1{e_R-e_R^{(q)}}e_{S-1}^{(q)}e_{T-2}^{(q)}
\\
\label{pf:Z.x.x2}
&=e_{T-1}Z_{RS1}-e_R Z_{S(T-1)1}+Z_{(R+1)S(T-1)}.
\end{align}
Replacing the pair $(R,T)$ with $(R+j,\,T-j)$ for $0\le j\le T-2$,
we obtain
\[
Z_{(R+j)S(T-j)}
=e_{T-j-1}Z_{(R+j)S1}
-e_{R+j}Z_{S(T-j-1)1}
+Z_{(R+j+1)S(T-j-1)}.
\]
Consequently, for $1\le k\le T-1$,
after~$k$ steps we have
\begin{align*}
Z_{RST}
&=\sum_{j=0}^{k-1}e_{T-1-j}Z_{(R+j)S1}
-\sum_{j=0}^{k-1}e_{R+j}Z_{S(T-1-j)1}
+Z_{(R+k)S(T-k)}.
\end{align*}
Taking $k=T-1$,
we find the last term becomes the summand for $j=T-1$ in the first sum.
Reindexing its terms and substituting the formula for~$Z_{RS1}$,
we obtain
\begin{align*}
Z_{RST}
&=\sum_{c=0}^{T-1}e_c Z_{(R+T-1-c)S1}
-\sum_{j=0}^{T-2}e_{R+j}Z_{S(T-1-j)1}
\\
&=\sum_{c=0}^{T-1}\sum_{b=0}^{S-1}
(N-c-2b)e_{(N-b-c)bc}
-\sum_{j=0}^{T-2}\sum_{c=0}^{T-2-j}
(S+T-2-j-2c)e_{(R+j)(S+T-2-j-c)c},
\end{align*}
which simplifies to \cref{eq:ZRST.e} by setting $a=N-b-c$ and $a=R+j$ in the
double sums,
respectively.

Finally,
the normalized contribution of the detached clique of~$S_{RST}^{100}$ is
\begin{align*}
\sum_{p\ne q}x_p e_{R-1}^{(p)}
&=\sum_p x_p e_{R-1}^{(p)}-x_q e_{R-1}^{(q)}
=Re_R-x_q e_{R-1}^{(q)}.
\end{align*}
Therefore,
\[
Y_{S^{100}_{RST}}
=\sum_q x_q
\brk1{Re_R-x_q e_{R-1}^{(q)}}
e_{S-1}^{(q)}e_{T-1}^{(q)}
=Re_R\sum_q x_q e_{S-1}^{(q)}e_{T-1}^{(q)}
-\sum_q x_q^2e_{R-1}^{(q)}e_{S-1}^{(q)}e_{T-1}^{(q)}.
\]
As in \cref{pf:Z.x.x2}, we obtain
\[
Y_{S^{100}_{RST}}
=Re_R Z_{ST1}-\brk1{e_R Z_{ST1}-Z_{(R+1)ST}}
=Z_{(R+1)ST}+(R-1)e_R Z_{ST1}.
\qedhere
\]
\end{proof}

We now derive the Schur expansion of~$G$ using Kostka numbers.

\begin{theorem}
\label{thm:Y.000.s}
Let $r\ge s\ge t\ge 2$ and $G=S_{rst}^{000}$. If $r\ge s+t-1$,
then
\[
Y_G
=\sum_{x,y\ge 0,\ 2x+y\le s+t-2}
C_{xy}
s_{3^x2^y1^{r+s+t-2-3x-2y}},
\]
where
\[
C_{xy}
=\sum_{i\in I_{xy}}
(t-2x-y+i)
\brk1{(r-i)(s-i)-(i-x+1)(x+y-i)},
\]
with
$I_{xy}=\{i\in\mathbb Z\colon \max(x,\,2x+y-t+1)\le i\le \min(s-1,\,x+y)\}$.
\end{theorem}

\begin{proof}
Suppose that $r\ge s+t-1$ and $\coeff{s_\lambda}{X_G}\ne 0$.
By \cref{lem:alpha-omega.s},
\[
\lambda_1
\le \alpha(G)
=3
\quad\text{and}\quad
\ell(\lambda)
\ge \chi(G)
=r.
\]
Thus we can suppose that $\lambda=3^x 2^y 1^z$,
where $x,y,z\ge 0$ and $x+y+z\ge r$. Then
\[
n
=\abs{V(G)}
=\abs{\lambda}
=3x+2y+z
\quad\text{and}\quad
\lambda'
=(x+y+z)(x+y)x.
\]
It follows that $2x+y=n-(x+y+z)\le s+t-2$.

Applying the Hall involution~$\omega$ to~$Z_{rst}$,
we obtain
\[
C_{xy}
=\coeff{s_\lambda}{Y_G}
=\coeff{s_{\lambda'}}{\omega Y_G}
=\coeff{s_{\lambda'}}{\omega Z_{rst}}.
\]
Taking $(N,R,S,T)=(n,r,s,t)$ in \cref{eq:ZRST.e} and using the Kostka number
$K_{\lambda,\mu}=\coeff{s_\lambda}{h_\mu}$ (see
\citet[Eq.~(6.6)(vi), p.~102, and Eq.~(6.4), p.~101]{Mac95B}),
we obtain
\[
C_{xy}
=\sum_{abc\vDash_w n,\ b<s,\ c<t}
(a-b)K_{\lambda',\,abc}
-\sum_{abc\vDash_w n,\ a\ge r,\ b\ge s}
(b-c)K_{\lambda',\,abc}.
\]
The Kostka number~$K_{\lambda',\,abc}$ counts semistandard Young tableaux of
shape~$\lambda'$ and content~$(a,b,c)$.
Let~$i$ count the entries~$2$ in the second row.
The entries~$1$ and~$2$ uniquely fill a subdiagram of shape $(a+b-i,\,i)$,
with all~$a$ entries~$1$ in the first row.
Thus
\[
x
\le i
\le x+y
\le a+b-i
\le n-2x-y.
\]
Conversely,
every such integer~$i$ gives exactly one tableau.
Since $a\ge r\ge s+t-1>x+y$ and $a+b+c=n$,
\[
K_{\lambda',\,abc}
=\abs{\{i\in\mathbb Z\colon \max(x,\,2x+y-c)\le i\le \min(b,\,x+y)\}}.
\]
Substituting this count gives
\[
C_{xy}
=\sum_{i=x}^{x+y}
\brk4{
\sum_{\substack{i\le b\le s-1
\\
2x+y-i\le c\le t-1
}}(n-2b-c)
-
\sum_{\substack{b+c\le s+t-2
\\
b\ge \max(i,\,s),\ c\ge 2x+y-i}}(b-c)
}.
\]

Let $x\le i\le x+y$.
The first inner sum is nonempty precisely when $i\in I_{xy}$.
For such~$i$, put $v=t-2x-y+i$.
It is routine to compute the inner sum,
using $n=r+s+t-2$,
\begin{align*}
\sum_{b=i}^{s-1}\sum_{c=2x+y-i}^{t-1}(n-2b-c)
&=(s-i)\brk4{v(r-i)+\binom v2}.
\end{align*}
For the same values of~$i$,
the second inner sum becomes
\[
\sum_{b+c\le s+t-2,\ b\ge s,\ c\ge 2x+y-i}(b-c)
=\sum_{p+q\le v-2,\ p, q\ge 0}(s-t+v+p-q)
=(s-t+v)\binom v2.
\]
If $i<s$ and $i<2x+y-t+1$, then $c\ge 2x+y-i\ge t$,
so both inner sums are empty.
For $i\ge s$,
only the second inner sum contributes,
and it becomes
\[
\sum_{\substack{b+c\le s+t-2
\\
b\ge i,\ c\ge 2x+y-i}}(b-c)
=\sum_{\substack{p+q\le s+t-2x-y-2
\\
p, q\ge 0}}(2i-2x-y+p-q)
=(2i-2x-y)\binom{s+t-2x-y}{2}.
\]
Since
\[
\sum_{i=s}^{x+y}(2i-2x-y)
=(s-x)\max(x+y-s+1,\,0),
\]
combining these results gives
\[
C_{xy}
=\sum_{i\in I_{xy}}
\brk4{(s-i)v(r-i)+(s-i)\binom v2-(s-t+v)\binom v2}
-D_{xy},
\]
where
\[
D_{xy}
=(s-x)\max(x+y-s+1,\,0)\binom{s+t-2x-y}{2}.
\]
To apply Abel summation (see \citep[Eq.~(2.56), p.~55]{GKP94B}),
put $A_i=(i-x+1)(x+y-i)$ and $v_i=t-2x-y+i$.
Since $A_i-A_{i-1}=2x+y-2i$,
we have the telescoping identity
\[
(2x+y-2i)\binom{v_i}{2}+A_i v_i
=A_i\binom{v_i+1}{2}-A_{i-1}\binom{v_i}{2}.
\]
Let $L=\max(x,\,2x+y-t+1)$ and $U=\min(s-1,\,x+y)$.
If $L\le U$,
summing over $i\in I_{xy}$ gives the boundary contribution
\[
A_U\binom{v_U+1}{2}-A_{L-1}\binom{v_L}{2}.
\]
The lower boundary vanishes:
if $L=x$, then $A_{L-1}=0$;
otherwise, $v_L=1$.
At the upper boundary,
if $U=s-1$, then $A_U\binom{v_U+1}{2}=D_{xy}$;
if $U=x+y<s-1$, then both $A_U$ and~$D_{xy}$ are zero.
Thus the boundary contribution equals~$D_{xy}$ and cancels the subtracted term,
yielding the stated formula.
If $L>U$,
the bounds $2x+y\le s+t-2$ and $t\le s$ imply $L\le s-1$,
so $U=x+y<s-1$ and again $D_{xy}=0$.
This completes the proof.
\end{proof}

Now we are in a position to determine the Schur positivity characterization.

\begin{theorem}
\label{thm:s+.000}
Let $r\ge s\ge t\ge 1$.
We have $S^{000}_{rst}\spos$ if and only if $r\ge st-1$.
\end{theorem}

\begin{proof}
Let $G=S^{000}_{rst}$. If $t=1$,
then $G\cong P^0(K_s,K_r)$, which is~$e$-positive by \cref{thm:barbell},
and hence $G\spos$.
Moreover, $r\ge s>s-1=st-1$.
Thus the assertion holds when $t=1$.

Suppose henceforth that $t\ge 2$. If $r\le s+t-2$,
then $r\le st-2$,
and $G\not\spos$ by \cref{prop:s+.nice,lem:nice-.000}.
If $s+t-1\le r\le st-2$, then \cref{thm:Y.000.s} gives
\[
C_{0(s+t-2)}
=r-st+1<0.
\]
In both cases,
$G\not\spos$.

It remains to suppose that $r\ge st-1$.
Then $r\ge s+t-1$.
For each summand in \cref{thm:Y.000.s},
put $j=2x+y-i$.
The index bounds give $0\le i\le s-1$,
$0\le j\le t-1$,
and $0\le x\le \min(i,j)$.
Thus
\[
(r-i)(s-i)-(i-x+1)(j-x)
\ge (r-s+1)-s(t-1)
=r-st+1
\ge 0.
\]
Since the remaining factor $t-j$ is positive,
every summand in the definition of $C_{xy}$ is nonnegative.
Hence $G\spos$, completing the proof.
\end{proof}

For example,
the claw $S_{222}^{000}\cong K_{1,3}$ satisfies $S_{222}^{000}\not\spos$,
since $r\le st-2$ as $2\le 2\cdot2-2$;
while the \emph{cricket graph}~$S_{322}^{000}$ satisfies $S_{322}^{000}\spos$,
since $r\ge st-1$ as $3=2\cdot2-1$,
but it is not~$e$-positive since $t=2$.
Indeed,
\[
X_{S_{322}^{000}}
=10e_5+14e_{41}-4e_{32}+4e_{31^2}
=4s_{31^2}+18s_{21^3}+24s_{1^5}.
\]

\section{\texorpdfstring{$e$- and Schur positivity of
$S_{rst}^{100}$ and~$S_{rst}^{010}$}{e- and Schur positivity of S100 and S010}}
\label{sec:S100:S010}

We classify the~$e$-positivity of these families using
\cref{lem:clique-spider.center},
and prove the Schur positivity using \cref{cor:s+.Hall.prefix}.

\begin{proposition}
\label{prop:Y**0.e}
Let $r,s,t$ be integers such that $\min(r,s)\ge t\ge 1$,
and $n=r+s+t-1$. Then
\begin{align*}
Y_{S^{100}_{rst}}
&=\sum_{abc\vDash_w n,\ a\ge r,\ b\le r-1,\ c\le t-1}
(a-b)e_{abc}
+\sum_{0\le c\le t-1}
\brk1{c+r(s+t-2c-2)} e_{r(s+t-1-c)c}
\\
\notag
&\qquad
-\sum_{abc\vDash_w n,\ a\ge r+1,\ b\ge s}
(a-c) e_{abc}.
\end{align*}
\end{proposition}
\begin{proof}
Let $G=S^{100}_{rst}$ and set $n=r+s+t-1$.
By \cref{eq:S100.ZB}, $Y_G=Z_{(r+1)st}+(r-1)e_r Z_{st1}$.

We use \cref{eq:ZRST.e} to simplify the two~$Z$-functions.
For~$Z_{(r+1)st}$,
its first sum satisfies $a\ge r+1$,
and we can recast it by subtracting the terms for $b\ge s$ from those for
all~$b$.
Writing $-(b-c)=(a-b)-(a-c)$ in the second sum which satisfies $b\ge s$,
we obtain
\[
Z_{(r+1)st}
=\sum_{abc\vDash_w n,\ a\ge r+1,\ c<t}(a-b)e_{abc}
-\sum_{abc\vDash_w n,\ a\ge r+1,\ b\ge s}(a-c)e_{abc}.
\]
The subtracted sum is exactly the last sum in the desired expression.
For~$Z_{st1}$,
we take $(R,S,T)=(s,t,1)$ in \cref{eq:ZRST.e}.
Then its second sum is empty.
Renaming the index~$b$ in its first sum as~$c$ gives
\[
e_r Z_{st1}
=\sum_{c=0}^{t-1}(s+t-1-2c)e_{r(s+t-1-c)c}.
\]

We now regroup the terms in these three sums.
Fix $c<t$ and put $m=n-c$.
We claim that
\[
\sum_{ab\vDash_w m,\ a\ge r+1}(a-b)e_{ab}
=\sum_{ab\vDash_w m,\ a\ge r,\ b<r}(a-b)e_{ab}
+(m-2r)e_{r(m-r)}.
\]
Indeed,
the terms with $a,b\ge r$ cancel by exchanging~$a$ and~$b$,
so
\[
\sum_{ab\vDash_w m,\ a\ge r}(a-b)e_{ab}
=\sum_{ab\vDash_w m,\ a\ge r,\ b<r}(a-b)e_{ab}.
\]
Removing the term with $a=r$ from the left side subtracts $(2r-m)e_{r(m-r)}$,
which proves the claim.

Multiplying by~$e_c$ and summing over $0\le c<t$ gives the first sum in the
statement and the boundary coefficient $n-2r-c$.
Adding $(r-1)e_r Z_{st1}$ changes this coefficient to
\[
n-2r-c+(r-1)(s+t-1-2c)
=c+r(s+t-2c-2).
\qedhere
\]
\end{proof}

Recall that the graph~$S^{100}_{rs1}$ is exactly the barbell $P^1(K_r,K_s)$.
Taking $t=1$ in \cref{prop:Y**0.e}, we obtain for $r\ge s\ge 1$,
\[
Y_{S^{100}_{rs1}}
=r(s-1) e_{rs}+\sum_{b=0}^{\min(r-1,\,s)}(r+s-2b)e_{(r+s-b)b}.
\]

We are then able to determine when these graphs are~$e$-positive.

\begin{theorem}
\label{thm:e+.**0}
Let $r\ge s\ge t\ge 1$. Let $G\in\{S^{100}_{rst},\,S^{010}_{rst}\}$.
Then~$G$ is~$e$-positive if and only if (i) $t=1$ or (ii) $t=2$ and $r=s$.
\end{theorem}

\begin{proof}
First we consider $G=S^{100}_{rst}$. If $t=1$,
then~$G$ is~$e$-positive by \cref{thm:barbell}. Let $t\ge 2$.
If $r=s$ and $t=2$, then $n=2r+1$,
and \cref{prop:Y**0.e} gives
\[
Y_G
=\sum_{i=0}^{r-1}
\brk1{(n-2i)e_{(n-i)i}+2(r-i)e_{(2r-i)i1}}
+(r^2-r-1) e_{(r+1)r}
+(r-1)^2 e_{r^21},
\]
which is~$e$-positive. If $r=s$ and $t\ge 3$,
then $n=2r+t-1$ and
\[
\coeff{e_{(r+t-2)(r+1)}}{Y_G}
=\begin{dcases*}
-(r+1), & if $t=3$,
\\
-n, & if $t\ge 4$;
\end{dcases*}
\]
this coefficient comes from the negative sum on the right side of the formula in
\cref{prop:Y**0.e},
and it cannot be canceled by any other term.
Otherwise, if $r\ge s+1$ and $t\ge 2$,
then
\[
\coeff{e_{(r+t-1)s}}{Y_G}
=(n-2s)-(r+t-1)
=-s<0.
\]
This proves the desired characterization for the~$e$-positivity of~$G$.

Now consider $G=S^{010}_{rst}\cong S^{100}_{srt}$.
The case $t=1$ has already been handled,
as has the case $r=s$ and $t=2$.
Suppose that $r>s$ and $t\ge 2$.
In view of \cref{prop:Y**0.e},
the only contributions to the coefficient
$\coeff{e_{n-r}e_r}{Y_{S^{100}_{srt}}}$ come from the final negative sum.
More precisely,
\[
\coeff{e_{n-r}e_r}{Y_{S^{100}_{srt}}}
=\begin{dcases*}
-n, & if $n-r>r$,
\\
r-n, & if $n-r\le r$.
\end{dcases*}
\]
Indeed,
the composition $(n-r,\,r,\,0)$ contributes $r-n$,
and when $n-r>r$,
the composition $(r,\,n-r,\,0)$ contributes an additional $-r$.
This coefficient is negative.
\end{proof}

The following result treats both possible orders of the first two clique sizes
and thereby gives a uniform proof for the two graph families.

\begin{theorem}
\label{thm:s+.100}
Let $r,s,t$ be integers such that $\min(r,s)\ge t\ge 1$.
Then $S^{100}_{rst}\spos$.
\end{theorem}

\begin{proof}
Let $G=S^{100}_{rst}$ and $n=r+s+t-1$. If $t=1$,
then $G=P^1(K_r,K_s)$,
so the result follows from \cref{thm:barbell}.
Suppose that $t\ge 2$. By \cref{prop:Y**0.e},
\[
Y_G
=\sum_{abc\vDash_w n,\ a\ge r,\ b\le r-1,\ c\le t-1}
(a-b)e_{abc}
+P-N,
\]
where
\[
P
=\sum_{0\le i\le t-1}a_i e_{\alpha_i}
\quad\text{and}\quad
N
=\sum_{j\ge 1,\ k\ge 0,\ j+k\le t-1}
b_{jk}e_{\beta_{jk}},
\]
with
\begin{align*}
&a_i
=t-1-i+r(2i+s-t),
\qquad
\alpha_i
=r(s+i)(t-1-i),
\\
&b_{jk}
=r+2j+k+1-t,
\qquad
\beta_{jk}
=(r+j)(s+k)(t-1-j-k).
\end{align*}
Since the first sum in~$Y_G$ is~$e$-positive,
it suffices to prove $P-N\spos$.

Put $\delta=r-s$.
After rearranging the entries of~$\alpha_i$ and~$\beta_{jk}$ into partitions and
padding with zeros to length three,
their largest and smallest parts are,
respectively,
\[
s+\max(\delta,\,i),
\quad t-1-i,
\qquad\text{and}\qquad
s+\max(\delta+j,\,k),
\quad t-1-j-k.
\]
Since the two partitions have the same size and at most three parts,
the dominance inequalities are equivalent to comparing their largest parts in
the usual direction and their smallest parts in the reverse direction.
Hence
\[
\alpha_i\unlhd\beta_{jk}
\iff
i
\le j+k
\quad\text{and}\quad
\max(\delta,\,i)
\le \max(\delta+j,\,k).
\]
Since $j\ge 1$,
the right side of the second inequality exceeds~$\delta$,
so that inequality is equivalent to $i\le \max(\delta+j,\,k)$.
Consequently,
\[
\alpha_i\unlhd\beta_{jk}
\iff
i
\le M_{jk},
\quad\text{where
$M_{jk}=\min\{j+k,\,\max(k,\,\delta+j)\}$}.
\]
In particular, $0\le M_{jk}\le t-1$.
For $0\le q\le t-1$, set
\[
D_q
=\{(j,k)\colon j
\ge 1,\ k
\ge 0,\ j+k
\le t-1,\ M_{jk}
\le q\}.
\]
By \cref{cor:s+.Hall.prefix},
it suffices to prove that $\Delta_q\ge 0$ for every~$q$,
where
\[
\Delta_q
=\sum_{i=0}^q a_i-
\sum_{(j,k)\in D_q}b_{jk}.
\]
Fix~$q$ and define
\[
c_q
=q+s-t,
\quad
n_k
=t-1-k,
\quad\text{and}\quad
m_k
=\min\{n_k,\,\max(q-k,\,q-\delta)\}.
\]
It is routine to check that
\[
0
\le m_k
\le \min(n_k,c_q)
\quad\text{and}\quad
D_q
=\{(j,k)\colon 0
\le k
\le q,\ 1
\le j
\le m_k\}.
\]
On the other hand,
the positive prefix sum can be distributed among the columns~$k$ as
\[
\sum_{i=0}^q a_i
=(q+1)\brk3{t-1-\frac q2+r(q+s-t)}
=\sum_{k=0}^q(n_k+rc_q).
\]
For each fixed~$k$,
the total negative weight in that column is
\[
\sum_{j=1}^{m_k}b_{jk}
=\sum_{j=1}^{m_k}(r+2j+k+1-t)
=m_k(r+m_k-n_k+1).
\]
It follows that
\[
\Delta_q
=\sum_{k=0}^q
\brk1{r(c_q-m_k)+(n_k-m_k)(m_k+1)}
\ge 0.
\]
This completes the proof.
\end{proof}

\begin{corollary}
\label{cor:s+.010}
For any integers $r\ge s\ge t\ge 1$,
$S^{010}_{rst}\spos$.
\end{corollary}

\begin{proof}
Since $S^{010}_{rst}\cong S^{100}_{srt}$,
we have $S^{010}_{rst}\spos$ by \cref{thm:s+.100}.
\end{proof}

\section{\texorpdfstring{$e$- and Schur positivity of
$S_{rst}^{001}$}{e- and Schur positivity of S001}}
\label{sec:S001}

Since $S^{001}_{rss}\cong S^{010}_{rss}$,
it remains to consider $s>t$.
We derive a signed~$e_I$-expansion for~$S^{001}_{rst}$ and use it to determine
both its~$e$- and Schur positivity.
The Schur classification is obtained from a common positive kernel together with
an explicit two-column obstruction.

\subsection{\texorpdfstring{Elementary expansion and
$e$-positivity}{Elementary expansion and e-positivity}}
\label{sec:S001.e+}

The identities in \cref{lem:clique-spider.center} also give a compact expansion
for this family.

\begin{proposition}
\label{prop:S001}
Let $r\ge s\ge t\ge 1$, $G=S^{001}_{rst}$,
and $n=r+s+t-1$. Then $Y_G=\sum_{k=r}^n Y_k e_k$,
where
\[
Y_k
=\sum_{i=n-k-t+1}^{n-k}(k-i)e_{i(n-k-i)}
+t(2k+t-n)e_{(n-t-k)t}
-\sum_{i=0}^{n-k-s}(k-i)e_{i(n-k-i)}.
\]
As a consequence, if $s>t$,
then $\coeff{e_{(n-s)s}}{X_G}\le -s$.
\end{proposition}
\begin{proof}
Permuting the clique parameters in \cref{eq:S100.ZB} gives
\[
Y_G
=Z_{rs(t+1)}+(t-1)e_t Z_{rs1}.
\]
In \cref{eq:ZRST.e} for~$Z_{rs(t+1)}$,
separate $c=t$ from the first sum.
The terms with $c=t$ are exactly $e_t Z_{rs1}$ by the $T=1$ case of
\cref{eq:ZRST.e}.
Writing $-(b-c)=(a-b)-(a-c)$ and combining the $(a-b)$-terms,
we obtain
\[
Y_G
=\sum_{abc\vDash_w n,\ a\ge r,\ c<t}
(a-b)e_{abc}
+te_t Z_{rs1}
-\sum_{abc\vDash_w n,\ a\ge r,\ b\ge s}
(a-c)e_{abc}.
\]
Rename the index~$a$ as~$k$. Then $r\le k\le n$.
Setting $b=i$ in the first sum and $c=i$ in the last gives respectively the
first and last sums in the desired formula for~$Y_k$.
As discussed,
the middle term is~$t$ times the $c=t$ term in the first sum of~$Z_{rs(t+1)}$
expanded in~\cref{eq:ZRST.e},
i.e.,
\[
te_t Z_{rs1}
=t\sum_{ab\vDash_w n-t,\ b<s}(a-b)e_{abt}
=\sum_{k=r}^{n-t}t(2k+t-n)e_{k(n-t-k)t}.
\]
This gives the middle term of~$Y_k$ and completes the expansion.

Suppose $s>t$.
The middle term of~$Y_k$ contributes nothing to the coefficient of~$e_{(n-s)s}$,
since it contains~$e_t$ and $t<s\le n-s$.
For $k=n-s$,
the first sum with $i=s$ contributes $n-2s$,
and the last sum with $i=0$ contributes $-(n-s)$.
Their total contribution is thus $-s$.
The only other possible value is $k=s$,
which requires $r=s$ because $k\ge r\ge s$.
If $r=s$,
this contribution is new only when $t>1$,
since otherwise $n-s=s$. In this case,
the first sum with $i=n-s$ contributes $2s-n$,
and the last sum with $i=0$ contributes $-s$,
giving $s-n$ in total. Therefore,
\[
\coeff{e_{(n-s)s}}{Y_G}
=-s+(s-n)\mathbf{1}(r
=s,\ t>1)
\le -s.
\]
Since the normalization factor is a positive integer,
the asserted inequality for~$X_G$ follows.
\end{proof}

Now we are able to give the~$e$-positivity characterization for the
graphs~$S^{001}_{rst}$.

\begin{corollary}
\label{cor:e+.S001}
Let $r\ge s\ge t\ge 1$.
The graph~$S^{001}_{rst}$ is~$e$-positive if and only if either $s=t=1$,
or $(r,s,t)=(2,2,2)$.
\end{corollary}

\begin{proof}
If $s=t$, then $S^{001}_{rss}\cong S^{010}_{rss}$,
and the result follows from \cref{thm:e+.**0}.
If $s>t$, non-$e$-positivity follows from \cref{prop:S001}.
\end{proof}

\subsection{Common kernel and Schur-positive cases}
\label{sec:S001.s+.positive}

The following theorem classifies the Schur positivity of~$S^{001}_{rst}$.

\begin{theorem}
\label{thm:s.001}
Let $r\ge s> t\ge 1$.
Then $S^{001}_{rst}\spos$ if and only if one of the
following conditions holds:
\begin{enumerate}
\item
$r\ge 2s-1$;
\item
$3t>r+s$;
\item
$(r,s)=(2t-1,\,t+1)$;
\item
$(r,s,t)=(6,6,4)$.
\end{enumerate}
\end{theorem}

A common kernel controls all four Schur-positive regimes in \cref{thm:s.001}.

\begin{lemma}[Common kernel expansion]
\label{lem:S001.kernel}
Let $r\ge s\ge t\ge 1$ and $n=r+s+t-1$. Then
\begin{equation}
\label{eq:S001.kernel}
Y_{S^{001}_{rst}}
=Q_{rst}
+\sum_{x,y\ge 0,\ x+y\le t-1}
\Psi(x,y)e_{r+y}
s_\mu,
\end{equation}
where
\[
Q_{rst}
=\sum_{\substack{0\le q\le p\le s-1
\\
m=p+q\le s+t-1}}
\brk4{
t(n-2m+t)\mathbf{1}(q\le t\le p)
+\sum_{i=\max(q,\,m-t+1)}^{p}
(n-m-i)
}e_{n-m}s_{2^q1^{p-q}}
\spos,
\]
$\mu=2^{t-1-x-y}1^{s-t+1+2x+y}$,
and
\begin{equation}
\label{eq:S001.Psi}
\Psi(x,y)
=t(r-s+1+2y)
+y(r-s+y+1)
+\binom{y}{2}
-(x+1)(s-t+1+x+y).
\end{equation}
\end{lemma}

\begin{proof}
Let $G=S_{rst}^{001}$.
For the index~$k$ in \cref{prop:S001} which satisfies $r\le k\le n$,
we introduce its complementary index $m=n-k$.
Then $Y_G=\sum_{m=0}^{s+t-1}(S_1+S_2+S_3)e_{n-m}$, 
where
\begin{align*}
S_1
&=\sum_{j=0}^{t-1}(n-2m+j)e_{j(m-j)}
=\sum_{j=0}^{t-1}
\sum_{q=0}^{\min(j,\,m-j)}
(n-2m+j)s_{2^q1^{m-2q}},
\\
S_2
&=t(n-2m+t)e_{t(m-t)}
=\sum_{q=0}^{\min(t,\,m-t)}
t(n-2m+t)s_{2^q1^{m-2q}},
\\
S_3
&=-\sum_{i=0}^{m-s}(n-m-i)e_{i(m-i)}
=-\sum_{i=0}^{m-s}
\sum_{q=0}^{\min(i,\,m-i)}
(n-m-i)s_{2^q1^{m-2q}}.
\end{align*}
In each row,
the second equality follows from \cref{prop:Pieri}.
Set $p=m-q$.
Then the restriction $q\le \min(u,\,m-u)$ becomes $q\le u\le p$,
where~$u$ is~$j$,
$t$, or~$i$, respectively.
As $m$ ranges from~$0$ to $s+t-1$,
the pair $(p,q)$ runs over all lattice points in the triangular area bounded by
\[
q
=0,\quad
q
=p,\quad\text{and}\quad
p+q
=s+t-1.
\]

For any fixed pair $(p,q)$ in this area,
extracting the coefficient of~$s_{2^q1^{p-q}}$ gives
\begin{align*}
\coeff{s_{2^q1^{p-q}}}{S_1}
&=\sum_{j=q}^{\min(p,\,t-1)}(n-2m+j),
\\
\coeff{s_{2^q1^{p-q}}}{S_2}
&=t(n-2m+t)\mathbf{1}(q\le t\le p),
\quad\text{and}
\\
\coeff{s_{2^q1^{p-q}}}{S_3}
&=-\sum_{i=q}^{\min(p,\,m-s)}(n-m-i).
\end{align*}
We claim that the upper bound $i\le p$ in the last sum is redundant.
In fact,
\[
2q
\le p+q
=m
\le s+t-1
\le 2s-1.
\]
Hence $q\le s-1$, and $m-s\le m-q-1<p$.
This proves the claim. Combining these coefficients,
we obtain
\[
Y_G
=\sum_{0\le q\le p,\ p+q\le s+t-1}
C(p,q)e_{n-p-q}s_{2^q1^{p-q}},
\]
where
\[
C(p,q)
=\sum_{j=q}^{\min(p,\,t-1)}(n-2m+j)
+t(n-2m+t)\mathbf{1}(q
\le t
\le p)
-\sum_{i=q}^{m-s}(n-m-i).
\]
Substituting $a=m-j$ and $a=m-i$ in the two sums,
respectively, gives
\[
\sum_{j=q}^{\min(t-1,\,p)}(n-2m+j)
-\sum_{i=q}^{m-s}(n-m-i)
=\sum_{a=\max(q,\,m-t+1)}^{p}(n-m-a)
-\sum_{a=s}^{p}(n-2m+a).
\]
The bounds already obtained give $\max(q,\,m-t+1)\le s$.
Therefore,
\[
C(p,q)
=t(n-2m+t)\mathbf{1}(q
\le t
\le p)
+\sum_{a=\max(q,\,m-t+1)}^{\min(p,\,s-1)}
(n-m-a)
-\sum_{a=s}^{p}(2a-m).
\]
Recall that every term except those in the last sum is nonnegative:
\begin{align*}
n-2m+t
&\ge n+t-2(s+t-1)
=r-s+1
>
0,\quad\text{and}
\\
n-m-a
&\ge r-(s-1)>0.
\end{align*}
Since the last sum is empty when $p<s$,
we find
\[
Q_{rst}
=\sum_{p<s}\sum_{\substack{0\le q\le p
\\
m=p+q\le s+t-1}}
C(p,q)e_{n-m}s_{2^q1^{p-q}},
\]
and $Q_{rst}\spos$ by \cref{prop:Pieri}.
Moreover,
\[
Y_G-Q_{rst}
=\sum_{\substack{0\le q\le p,\ p\ge s
\\
m=p+q\le s+t-1}}
C(p,q)e_{n-m}s_{2^q1^{p-q}},
\]
where $C(p,q)$ is reduced under the condition $p\ge s$ as
\[
C(p,q)
=t(n-2m+t)
+\sum_{a=m-t+1}^{s-1}
(n-m-a)
-\sum_{a=s}^{p}(2a-m).
\]
To make the summation indices start at~$0$,
we introduce
$x=p-s$ and $y=n-m-r$.
Then $x,y\ge 0$ and
\[
x+y
=n-q-r-s
=t-1-q
\le t-1.
\]
It is then routine to simplify
\[
\sum_{\substack{0\le q\le p,\ p\ge s
\\
m=p+q\le s+t-1}}
C(p,q)e_{n-m}s_{2^q1^{p-q}}
=\sum_{x,y\ge 0,\ x+y\le t-1}
\Psi(x,y)e_{r+y}
s_{2^{t-1-x-y}1^{s-t+1+2x+y}},
\]
as desired. This completes the proof.
\end{proof}

%
%
%

The common kernel yields the positive cases in the order in which they appear in
\cref{thm:s.001}.

\begin{proof}[Proof of the ``if'' part of \cref{thm:s.001}]
Let $G=S^{001}_{rst}$ and $n=r+s+t-1$.
Let~$Q_{rst}$ and~$\Psi$ be as in \cref{lem:S001.kernel}.
It is routine to check that,
for $x,y\ge 0$ with $x+y\le t-1$,
\begin{align}
\label{eq:S001.Psi.dx}
\Psi(x+1,\,y)-\Psi(x,y)
&=t-s-2x-y-3<0,
\\
\label{eq:S001.Psi.dy}
\Psi(x,\,y+1)-\Psi(x,y)
&=r-s+2t-x+3y+1>0.
\end{align}
\noindent\emph{Case 1. Suppose that $r\ge 2s-1$.} The monotonicity above gives
\[
\Psi(x,y)
\ge \Psi(t-1,\,0)
=t(r-2s+1)
\ge 0.
\]
Therefore,
$Y_G\spos$ by \cref{lem:S001.kernel}.

\noindent\emph{Case 2. Suppose that $3t>r+s$.} Set
$C_\lambda
=\coeff{s_\lambda}{(Y_G-Q_{rst})}$.
By \cref{prop:Pieri},
every Schur function occurring in $Y_G-Q_{rst}$,
or equivalently in $e_{r+y}s_\mu$,
has first part at most three.
Fix $\lambda\vdash n$ with $\lambda_1\le 3$, and write
$\lambda'
=abc$, where $a
\ge b
\ge c
\ge 0$.
For a term in \cref{eq:S001.kernel} indexed by $(x,y)$,
put $z=x+y$. Then $\mu'=(s+x)(t-1-z)$.
By \cref{prop:Pieri},
\[
\coeff{s_\lambda}{e_{r+y}s_\mu}
=\mathbf{1}(\text{$\lambda'/\mu'$ is a horizontal strip})
=\mathbf{1}(a
\ge s+x
\ge b
\ge t-1-z
\ge c).
\]
By \cref{lem:S001.kernel},
\begin{equation}
\label{eq:S001.Psi.Pieri}
C_\lambda
=\sum_{z=\max(0,\,t-1-b)}^{t-1-c}
\sum_{x=\max(0,\,b-s)}^{\min(a-s,\,z)}
\Psi(x,\,z-x).
\end{equation}

Let $\alpha=r-s$ and $\delta=3t-r-s-1$. Then $\alpha,\delta\ge 0$.

\noindent\emph{Case 2-1. If $b\le r$,} then
\[
a-s
=t-1-c+(r-b)
\ge z.
\]
Set $d=\max(0,\,b-s)$. If $z<d$,
then the inner sum is empty.
Suppose that $z\ge d$, and set $\Delta=z-d$.
Then $\Delta\ge 0$. A direct summation gives
\[
\sum_{x=d}^{z}\Psi(x,\,z-x)
=\frac{\Delta+1}{2}
\brk2{
(s-t)(4h+2d+3\Delta+2)
+2(\alpha+1+\Delta)\delta
+h(2h+4d+3\Delta+4)
},
\]
where $h=\alpha-d\ge 0$.
Thus the sum above is nonnegative and so is \cref{eq:S001.Psi.Pieri}.

\noindent\emph{Case 2-2. If $b>r$,} then we put $B=b-r$ and $u=a-b$.
Then $B\ge 1$ and $u\ge 0$. Moreover,
\[
b-s
=\alpha+B,
\qquad
a-s
=\alpha+B+u,
\qquad
t-1-c
=\alpha+2B+u.
\]
The two lower bounds in \cref{eq:S001.Psi.Pieri} therefore reduce to~$0$ and
$\alpha+B$,
respectively.
The inner sum is empty when $z<\alpha+B$.
Setting $i=x-(\alpha+B)$ and $y=z-x$ recasts \cref{eq:S001.Psi.Pieri} as
\[
C_\lambda
=\sum_{i=0}^{u}
\sum_{y=0}^{B+u-i}
\Psi(\alpha+B+i,\,y).
\]
A direct Maple evaluation of the double sum,
followed by simplification and factorization,
gives
\[
C_\lambda
=\frac{u+1}{12}
\brk2{
3(c+\delta)
\brk1{\alpha(2B+u+2)+2B(B+u+2)}
+(u+2)\brk1{3(c+tu)+\delta(u+3)}
}
\ge 0.
\]
Here we used $t=2B+\alpha+c+u+1$.

Thus $C_\lambda\ge 0$ in both subcases.
Since $Q_{rst}\spos$, we have $Y_G\spos$,
proving Case~2.

\noindent\emph{Case 3. Suppose that $(r,s)=(2t-1,\,t+1)$.} Then $t\ge 2$. For
$y=0$,
\cref{eq:S001.Psi} becomes
\[
\Psi(x,0)
=t(t-1)-(x+1)(x+2).
\]
It follows that
$\Psi(t-2,\,0)=0$
and
$\Psi(t-1,\,0)=-2t$.
If $y\ge 1$, then $x\le t-2$,
so the monotonicity in \cref{eq:S001.Psi.dx,eq:S001.Psi.dy} gives
\[
\Psi(x,y)
\ge \Psi(t-2,1)
=2t+1.
\]
Thus the only negative coefficient in \cref{eq:S001.kernel} is indexed by
$(t-1,\,0)$.
The two factors indexed by $(t-1,\,0)$ and $(t-2,\,1)$ coincide since
$e_{2t-1}s_{1^{2t}}
=e_{2t-1}e_{2t}
=e_{2t}s_{1^{2t-1}}$.
Their combined coefficient is
\[
\Psi(t-1,\,0)+\Psi(t-2,\,1)
=-2t+(2t+1)
=1.
\]
All remaining terms have nonnegative Schur coefficients,
so $Y_G\spos$,
proving Case~3.

\noindent\emph{Case 4. Suppose that $(r,s,t)=(6,6,4)$.} By
\cref{prop:S001,prop:Pieri}, one may compute that
\begin{align*}
Y_G
={}&
4s_{3^{4}21}
+16s_{3^{4}1^{3}}
+s_{3^{3}2^{3}}
+19s_{3^{3}2^{2}1^{2}}
+55s_{3^{3}21^{4}}
+45s_{3^{3}1^{6}}
+7s_{3^{2}2^{4}1}
\\
&+52s_{3^{2}2^{3}1^{3}}
+127s_{3^{2}2^{2}1^{5}}
+132s_{3^{2}21^{7}}
+96s_{3^{2}1^{9}}
+22s_{32^{5}1^{2}}
+110s_{32^{4}1^{4}}
\\
&+242s_{32^{3}1^{6}}
+274s_{32^{2}1^{8}}
+259s_{321^{10}}
+175s_{31^{12}}
+50s_{2^{6}1^{3}}
+200s_{2^{5}1^{5}}
\\
&+410s_{2^{4}1^{7}}
+484s_{2^{3}1^{9}}
+505s_{2^{2}1^{11}}
+448s_{21^{13}}
+288s_{1^{15}}\spos.
\end{align*}
This proves Case~4 and completes the proof.
\end{proof}

\subsection{Two-column obstruction and failure of Schur positivity}
\label{sec:S001.s+.negative}

For the remainder of \cref{sec:S001},
let $G=S^{001}_{rst}$ and $n=r+s+t-1$, and set
$\lambda=2^B1^{n-2B}$
and
$c_\lambda=\coeff{s_\lambda}{Y_G}$,
where
\begin{equation}
\creflabel[def]{def:B.001}
B
=\min\brk3{s+t-1,\ \floor*{\frac n2}}
=\begin{dcases*}
\floor*{\frac n2}, & if $r\le s+t-1$,
\\
s+t-1, & if $r\ge s+t-1$.
\end{dcases*}
\end{equation}
For triples $(r,s,t)$ in the negativity regime of \cref{thm:s.001},
it suffices to show that $c_\lambda<0$.

The common kernel gives a direct extraction of this coefficient.

\begin{theorem}
\label{thm:S001-two-column-coefficient}
Let $r\ge s\ge t\ge 1$, $G=S^{001}_{rst}$,
and $n=r+s+t-1$,
and let~$B$ be defined by \cref{def:B.001}.
Set $\lambda=2^B1^{n-2B}$ and $c_\lambda=\coeff{s_\lambda}{Y_G}$.
Put $d=r+s-3t$ and assume that $d\ge 0$.
Then
\[
c_\lambda
=\begin{dcases*}
-t(2s-r-1),
& if $r\ge s+t-1$,
\\
-(j+1)(s-t-j)(2t-s+j+1),
& if $r\le s+t-2$ and $d=2j+1$ is odd,
\\
(d+1)(s-t-j)^2-(d+1)(s-t)t+(dj+1)t,
& if $r\le s+t-2$ and $d=2j$ is even.
\end{dcases*}
\]
\end{theorem}
\begin{proof}
Since $s+t-1\ge s$ and $n\ge 2s$, we have $B\ge s$.

We claim that $\coeff{s_\lambda}{Q_{rst}}=0$ in \cref{lem:S001.kernel}.
Indeed,
every term of~$Q_{rst}$ is a scalar~$e$-multiple of~$s_{2^q1^{p-q}}$.
By \cref{prop:Pieri},
the multiplicity of the part~$2$ in any partition in the Schur expansion of
an~$e$-multiple of $s_{2^q1^{p-q}}$ is at most the sum of multiplicities of~$2$
and~$1$,
that is, $p$,
which is less than~$s$. This contradicts $B\ge s$,
and proves the claim.

For the remaining sum in \cref{eq:S001.kernel},
\cref{prop:Pieri} gives $B\le s+x\le n-B$. As a consequence,
\begin{equation}
\label{pf:c.001}
c_\lambda
=\sum_{x=B-s}^{n-B-s}
\sum_{y=0}^{t-1-x}\Psi(x,y).
\end{equation}
If $r\ge s+t-1$, then $B=s+t-1$ by definition,
and
\[
c_\lambda
=\Psi(t-1,\,0)
=t(r-2s+1).
\]
Suppose that $r\le s+t-2$. For any integer~$M\ge 0$,
summing \cref{eq:S001.Psi} over~$y$ gives
\[
T(x,M)
:=\sum_{y=0}^{M}\Psi(x,y)
=(M+1)\brk3{
t(r-s+1)-(x+1)(s-t+1+x)
+\frac M2(r-s+2t-x+M)}.
\]
Set $j=\floor{d/2}$ and $a=s-t-j$. Since $r\le s+t-2$,
we have
\[
2j
\le d
=r+s-3t
\le 2(s-t-1),
\]
which implies $a=s-t-j\ge 1$.

If $d=2j+1$, then $n=r+s+t-1=d+4t-1=4t+2j$,
and $B=2t+j$ by definition.
It follows that $n-B=B$ and the outer sum for \cref{pf:c.001} has the unique
term
\[
x
=B-s
=2t+j-(t+j+a)
=t-a.
\]
Using $r=d+3t-s=2j+1+3t-s=2t+j+1-a$, we can evaluate
\[
c_\lambda
=T(t-a,\,a-1)
=-(j+1)a(t+1-a).
\]
Note that the proved condition $a\ge 1$ guarantees the validity of the second
parameter of $T(t-a,\,a-1)$.

If $d=2j$, then $n=4t+2j-1$,
and $B=2t+j-1$ by definition.
It follows that $n-B=B+1$ and the outer sum for \cref{pf:c.001} consists of two
terms: $x=B-s=t-a-1$ and $x=t-a$.
Using $r=2t+j-a$, we can evaluate
\begin{align*}
c_\lambda
&=T(t-a-1,\,a)+T(t-a,\,a-1)
\\
&=(a+1)\brk3{(1-j)t+\brk2{j+\frac12}a}
+a\brk3{-(j+2)t+\brk2{j+\frac12}(a-1)}
\\
&=(2j+1)a^2-\brk1{(2j+1)a+j-1}t.
\end{align*}
Substituting $a=s-t-j$ and $d=2j$ gives the last asserted formula.
\end{proof}

Now we can complete the proof of \cref{thm:s.001}.

\begin{proof}[Proof of the ``only if'' part of \cref{thm:s.001}]
We prove the contrapositive. Let $r\ge s>t\ge 1$,
put $d=r+s-3t$,
and suppose that all the following conditions hold:
\begin{enumerate}
\item
\creflabel[itm]{itm:664}
$(r,s,t)\ne (6,6,4)$;
\item
\creflabel[itm]{itm:r<=2s-2}
$r\le 2s-2$;
\item
\creflabel[itm]{itm:d>=0.s>=t+2}
either $d\ge 1$, or $d=0$ and $s\ge t+2$.
\end{enumerate}
These conditions say precisely that none of the four cases in the statement of
\cref{thm:s.001} holds.
Indeed, if $d=0$,
then $s=t+1$ is equivalent to $(r,s)=(2t-1,\,t+1)$.
By the definition of~$c_\lambda$,
it is enough to prove that $c_\lambda<0$.
First suppose that $r\ge s+t-1$. By \cref{thm:S001-two-column-coefficient},
\[
c_\lambda
=-t(2s-r-1)
<0
\]
by \cref{itm:r<=2s-2}.

Assume instead that $r\le s+t-2$. By \cref{itm:d>=0.s>=t+2},
we have $d\ge 0$. Write $d=2j+1$ or $d=2j$,
according to the parity of~$d$.
Then
\begin{equation}
\creflabel[ineq]{pf:j>=0}
j
\ge 0.
\end{equation}
Moreover,
\[
2j
\le d
=r+s-3t
\le (s+t-2)+s-3t
=2(s-t-1).
\]
We keep using the symbol $a=s-t-j$.
The preceding inequality implies
\begin{equation}
\creflabel[ineq]{pf:a>=1}
a
\ge 1.
\end{equation}
On the other hand,
\begin{equation}
\creflabel[ineq]{pf:a<=t+1-a}
a
=s-t-j
\le r-t-j
=(d-s+3t)-t-j
\le 2t-s+j+1
=t+1-a.
\end{equation}

Assume first that $d=2j+1$. By \cref{thm:S001-two-column-coefficient},
\[
c_\lambda
=-(j+1)(s-t-j)(2t-s+j+1)
=-(j+1)a(t+1-a).
\]
By \cref{pf:j>=0,pf:a>=1,pf:a<=t+1-a}, we have $j+1>0$,
$a>0$, and $t+1-a>0$.
Therefore $c_\lambda<0$.

Suppose that $d=2j$. Then
\begin{equation}
\creflabel[ineq]{pf:t>=2a}
0
\le r-s
=(d-s+3t)-s
=t-2a.
\end{equation}
In particular, $t\ge 2a$.
By \cref{thm:S001-two-column-coefficient},
\[
c_\lambda
=(d+1)a^2
-\brk1{(d+1)(a+j)-2j^2-1}t.
\]
Set $H=(d+1)(a+j)-2j^2-1$. Using $d=2j$ and \cref{pf:a>=1},
we get
\[
H
\ge (2j+1)(j+1)-2j^2-1
=3j
\ge 0.
\]
Together with \cref{pf:t>=2a}, this gives
\[
c_\lambda
=(d+1)a^2-Ht
\le (d+1)a^2-2aH
=-\brk1{(2a+2)j+a-2}a.
\]
If $j\ge 1$, then $(2a+2)j+a-2\ge 3a>0$,
and hence $c_\lambda<0$.

If $j=0$, equivalently $d=0$,
then $a=s-t$. By \cref{itm:d>=0.s>=t+2},
we have $a\ge 2$. By \cref{thm:S001-two-column-coefficient,pf:t>=2a},
\begin{equation}
\creflabel[ineq]{pf:Crst<=0}
c_\lambda
=a^2-(a-1)t
\le a^2-2a(a-1)
=a(2-a)
\le 0.
\end{equation}
Equality in \cref{pf:Crst<=0} can occur only when $a=2$ and $t=2a$.
In that case,
\[
t
=4,\quad
s
=t+a
=6,\quad\text{and}\quad
r
=d-s+3t
=6.
\]
Thus $(r,s,t)=(6,6,4)$, contradicting \cref{itm:664}.
Therefore equality never occurs in \cref{pf:Crst<=0},
and so $c_\lambda<0$.

In all cases, $c_\lambda<0$.
Thus the Schur coefficient of $X_{S^{001}_{rst}}$ indexed by~$\lambda$ is
negative.
Consequently,
$S^{001}_{rst}\not\spos$.
This proves the ``only if'' part and completes the proof of \cref{thm:s.001}.
\end{proof}

\section{\texorpdfstring{The spiders $S(a,b,2)$ are Schur positive}
{The spiders S(a,b,2) are Schur positive}}
\label{sec:Sab2}

We make further progress on the~$e$-positivity of the spiders $S(a,b,2)$ and
establish Schur positivity for the entire family.
The~$e$-positivity argument extracts a negative coefficient from a path
expansion,
while the Schur-positivity proof uses a path--clique bootstrap from selected
clique-spider base graphs.

\citet[Lem.~4.4]{Zhe22} gave a formula for the chromatic symmetric function of
$S(a,b,2)$ in terms of those of paths.

\begin{proposition}[{\citeauthor{Zhe22}}]
\label{prop:X.Sab2}
Let $a\ge b\ge 2$ and $n=a+b+3$. Then
\[
X_{S(a,b,2)}
=X_{P_n}+e_1 X_{P_{n-1}}+2e_2 X_{P_{n-2}}
-X_{P_{a+1}}X_{P_{b+2}}-X_{P_{a+2}}X_{P_{b+1}}.
\]
\end{proposition}

\citet[Thm.~4.1 and Cor.~4.4 (2)--(4)]{WW23-DAM} obtained some necessary
conditions for the~$e$-positivity of the spiders $S(a,b,2)$.

\begin{theorem}[{\citeauthor{WW23-DAM}}]
\label{thm:WW.e+.Sab2}
Let $a\ge b\ge 2$.
Suppose that the spider $S(a,b,2)$ is~$e$-positive.
Then either $3\mid b$ or $(a,b)\equiv(0,1)\pmod3$.
Moreover, if $b\le 11$,
then $S(a,b,2)$ is~$e$-positive if and only if
$(a,b)\in\{(5,3),\,(6,4),\,(12,4),\,(8,6),\,(9,7)\}$.
\end{theorem}

We now settle the case $3\nmid b$.

\begin{theorem}
\label{thm:e+.Sab2.nonzero.residue}
Suppose that $a\ge b\ge 2$ and $3\nmid b$.
Then $S(a,b,2)$ is~$e$-positive if and only if
\[
(a,b)\in\{(6,4),\,(12,4),\,(9,7)\}.
\]
\end{theorem}

\begin{proof}
The~$e$-positivity of the three listed spiders follows from
\cref{thm:WW.e+.Sab2}.
It remains to prove the converse.
Let $G=S(a,b,2)$. By \cref{thm:WW.e+.Sab2},
we can suppose that $b=3k+1$ and $a=3A$ for some $k,A\in\mathbb Z$ such that
$k\ge 4$ and $A\ge 5$.
We shall show that~$G$ is not~$e$-positive.

Let $n=\abs{V(G)}$. Then $n=a+b+3=3A+3k+4$.
Set $\lambda=5^23^{A+k-2}\vdash n$. By \cref{prop:X.Sab2},
\begin{align*}
\coeff{e_\lambda}{X_G}
&=\coeff{e_\lambda}{X_{P_n}}
-\coeff{e_\lambda}{X_{P_{a+1}}X_{P_{b+2}}}
-\coeff{e_\lambda}{X_{P_{a+2}}X_{P_{b+1}}}
\\
&=\coeff{e_{5^23^{A+k-2}}}{X_{P_n}}
-\coeff{e_{5^23^{A-3}}}{X_{P_{a+1}}}
\coeff{e_{3^{k+1}}}{X_{P_{b+2}}}
-\coeff{e_{53^{A-1}}}{X_{P_{a+2}}}
\coeff{e_{53^{k-1}}}{X_{P_{b+1}}}.
\end{align*}
Using the path formula in \cref{thm:barbell},
we obtain for $m\ge 1$,
\begin{align*}
\coeff{e_{3^m}}{X_{P_{3m}}}
&=3\cdot2^{m-1},
\\
\coeff{e_{53^m}}{X_{P_{3m+5}}}
&=(6m+5)2^m,
\quad\text{and}
\\
\coeff{e_{5^23^m}}{X_{P_{3m+10}}}
&=4(m+1)(3m+5)2^m.
\end{align*}
Since $k\ge 4$ and $A\ge 5$,
we can continue the calculation above as
\begin{align*}
\coeff{e_\lambda}{X_G}
&=\coeff{e_{5^23^{A+k-2}}}{X_{P_n}}
-\coeff{e_{5^23^{A-3}}}{X_{P_{a+1}}}
\coeff{e_{3^{k+1}}}{X_{P_{b+2}}}
-\coeff{e_{53^{A-1}}}{X_{P_{a+2}}}
\coeff{e_{53^{k-1}}}{X_{P_{b+1}}}
\\
&=4(A+k-1)(3A+3k-1)2^{A+k-2}
-12(A-2)(3A-4)2^{A+k-3}
-(6A-1)(6k-1)2^{A+k-2}
\\
&=2^{A+k-2}(-6k^2-24kd-6d^2+16k+38d-1),
\end{align*}
where $d=A-k-1$. Since $a\ge b$,
we find $d\ge 0$, and
\begin{align*}
-6k^2-24kd-6d^2+16k+38d-1
&=-6k^2+16k-1-d(24k+6d-38)
\\
&\le -6k^2+16k-1
\le -8k-1
<0.
\end{align*}
This completes the proof.
\end{proof}

We now turn to Schur positivity.
To derive the aforementioned path--clique bootstrap,
we need the following formula of \citet[Prop.~3.1]{QTW26}.

\begin{theorem}[\citeauthor{QTW26}]
\label{thm:KPG}
Let~$H$ be a rooted graph.
Then for any $l\ge 0$ and $m\ge 1$,
\[
X_{P^l(H,K_m)}
=(m-1)!\sum_{i=0}^{m-1}(1-i)e_i X_{H^{l+m-1-i}}.
\]
\end{theorem}

Here is the path--clique bootstrap lemma.

\begin{lemma}
\label{lem:path-clique.bootstrap}
Let~$H$ be a rooted graph and let $L\ge 1$.
If $H^j\spos$ for $0\le j<L$
and $P^0(H,\,K_{L+1})\spos$,
then $H^L\spos$.
\end{lemma}

\begin{proof}
Taking $(l,m)=(0,\,L+1)$ in \cref{thm:KPG} and isolating~$X_{H^L}$ gives
\[
X_{H^L}
=\frac{X_{P^0(H,K_{L+1})}}{L!}
+\sum_{i=2}^{L}(i-1)e_i X_{H^{L-i}}
\spos.
\]
The final relation follows from the hypotheses and \cref{prop:Pieri}.
\end{proof}

We next prove Schur positivity for the clique-extended graphs used below.

\begin{lemma}
\label{lem:s+.Crs}
For any integers $r\ge 2$ and $1\le s\le r+1$,
$S^{102}_{rs1}\spos$.
\end{lemma}

\begin{proof}
Let $G=S^{102}_{rs1}$. Let $H=S^{100}_{rs1}$,
rooted at its center. For $(r,s)=(5,4)$,
these graphs are illustrated in \cref{fig:G-H-541}.

\begin{figure}[ht]
\centering
\begin{tikzpicture}[x=0.85cm,y=0.85cm]
\begin{scope}
\coordinate (c) at (0,0);
\coordinate (a0) at (-1,0);
\coordinate (a1) at (-1.45,0.75);
\coordinate (a2) at (-2.2,0.5);
\coordinate (a3) at (-2.2,-0.5);
\coordinate (a4) at (-1.45,-0.75);
\coordinate (b1) at (0,1);
\coordinate (b2) at (0.85,1.2);
\coordinate (b3) at (1.45,0.45);
\coordinate (d1) at (0.85,-0.3);
\coordinate (d2) at (1.7,-0.6);
\draw[edge] (c)--(a0);
\draw[edge] (a0)--(a1)--(a2)--(a3)--(a4)--cycle;
\draw[edge] (a0)--(a2) (a0)--(a3) (a1)--(a3)
(a1)--(a4) (a2)--(a4);
\draw[edge] (c)--(b1)--(b2)--(b3)--cycle;
\draw[edge] (c)--(b2) (b1)--(b3);
\draw[edge] (c)--(d1)--(d2);
\foreach \vertex in {c,a0,a1,a2,a3,a4,b1,b2,b3,d1,d2}
\node[ball] at (\vertex) {};
\node[font=\scriptsize,anchor=south east,inner sep=0pt]
at ($(c)+(-0.05,0.05)$) {$c$};
\node[font=\small] at (0,-1.15) {(a) $G=S^{102}_{541}$};
\end{scope}

\begin{scope}[xshift=6cm]
\coordinate (c) at (0,0);
\coordinate (a0) at (-1,0);
\coordinate (a1) at (-1.45,0.75);
\coordinate (a2) at (-2.2,0.5);
\coordinate (a3) at (-2.2,-0.5);
\coordinate (a4) at (-1.45,-0.75);
\coordinate (b1) at (0,1);
\coordinate (b2) at (0.85,1.2);
\coordinate (b3) at (1.45,0.45);
\draw[edge] (c)--(a0);
\draw[edge] (a0)--(a1)--(a2)--(a3)--(a4)--cycle;
\draw[edge] (a0)--(a2) (a0)--(a3) (a1)--(a3)
(a1)--(a4) (a2)--(a4);
\draw[edge] (c)--(b1)--(b2)--(b3)--cycle;
\draw[edge] (c)--(b2) (b1)--(b3);
\foreach \vertex in {c,a0,a1,a2,a3,a4,b1,b2,b3}
\node[ball] at (\vertex) {};
\node[font=\scriptsize,anchor=south east,inner sep=0pt]
at ($(c)+(-0.05,0.05)$) {$c$};
\node[font=\small] at (0,-1.15) {(b) $H=S^{100}_{541}$};
\end{scope}
\end{tikzpicture}
\caption{For $(r,s)=(5,4)$, the graphs $G=S^{102}_{541}$ and
$H=S^{100}_{541}$.  The center~$c$ of~$H$ is the root of~$H$, and
$G\cong H^2$.}
\label{fig:G-H-541}
\end{figure}
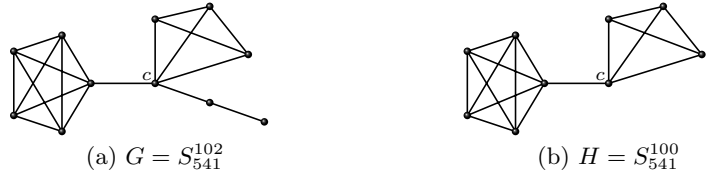

For $l\ge 0$ and $m\ge 1$,
subtracting two instances of \cref{thm:KPG} gives
\begin{equation}
\label{eq:path-clique.step}
mX_{P^{l+1}(H,K_m)}
-X_{P^l(H,K_{m+1})}
=m!(m-1)e_m X_{H^l}.
\end{equation}
Taking $(l,m)=(0,2)$ in \cref{eq:path-clique.step}, and observing that
$P^1(H,K_2)\cong H^2=G$,
gives
\[
2X_G
=X_{J}+X_{K_2}X_H,
\]
where $J=S^{100}_{rs3}$.
It remains to check that $H\spos$ and $J\spos$.
Since $\min(r,s)\ge 1$,
we have $H\spos$ by \cref{thm:s+.100}.
We next consider~$J$. If $\min(r,s)\ge 3$,
then $J\spos$ by \cref{thm:s+.100}.
If $s\le 2$, then $J\cong S^{100}_{r3s}$,
and \cref{thm:s+.100} again applies.
The only remaining possibility is $(r,s)=(2,3)$.
In this case, $J\cong S^{001}_{332}$,
so $J\spos$ by \cref{thm:s.001},
using the regime $(r,s)=(2t-1,\,t+1)$ with $t=2$.
This completes the proof.
\end{proof}

The next lemma gives the Schur positivity of further clique-extended graphs.

\begin{lemma}
\label{lem:s+.Drq}
For any integers $r\ge 2$ and $0\le q\le r$,
$S^{1q2}_{r11}\spos$.
\end{lemma}

\begin{proof}
Let $H=S^{102}_{r11}$, rooted at the center.
Then $H^q=S^{1q2}_{r11}$ for any $q\in\mathbb Z_{\ge 0}$.
By \cref{lem:s+.Crs}, $H^0\spos$,
and $P^0(H,K_{q+1})=S^{102}_{r(q+1)1}\spos$ for any
$0\le q\le r$.
Applying \cref{lem:path-clique.bootstrap},
we obtain $H^j\spos$ successively for $1\le j\le q$.
\end{proof}

\begin{theorem}
\label{thm:s+.Sab2}
For any integers $a,b\ge 1$,
$S(a,b,2)\spos$.
\end{theorem}

\begin{proof}
By symmetry, suppose that $a\ge b$.
We proceed by induction on $a+b$.
If $b=1$, then $S(a,b,2)\cong S(a,2,1)$,
so $S(a,b,2)\spos$ by \citet[Thm.~4.2]{TW23X}.
Suppose that $b\ge 2$. Let $H=S(1,b,2)$,
rooted at the leaf of its shortest leg.
Then $H^j=S(j+1,\,b,\,2)$.
For any $0\le j\le a-2$,
we have $H^j\spos$ by the induction hypothesis.
On the other hand,
we have $P^0(H,K_a)=S^{1b2}_{a11}\spos$ by \cref{lem:s+.Drq}.
Applying \cref{lem:path-clique.bootstrap} with $L=a-1$ gives
$S(a,b,2)=H^{a-1}\spos$.
\end{proof}

Using the clique-spider Schur-positivity results established in this paper,
\citet[Thm.~1.1]{WW26Sabc} obtained a complete characterization
of Schur positivity for all three-legged spiders.

\section*{Disclosure of artificial intelligence use}

The authors used the ChatGPT~5 series,
ChatGPT~6 Pro,
and OpenAI Codex in developing and revising this paper.
The ChatGPT~5 series assisted with constructing \cref{thm:equivalence}.
All other GPT contributions described below were provided by ChatGPT~6 Pro,
hereafter called GPT.

GPT shortened the $36$-page
\href{https://arxiv.org/pdf/2608.22184v1}{arXiv version~1}
by consolidating the proof of \cref{thm:equivalence} and 
simplifying several other proofs:
\cref{lem:clique-spider.center} replaces composition truncation and multiple-sum cancellation
in the proofs of \cref{prop:Y**0.e,prop:S001};
the Abel summation in the proof of \cref{thm:Y.000.s} 
replaces bialternant extraction
and two-variable generating functions;
direct Pieri extraction in the proof of \cref{thm:S001-two-column-coefficient}
replaces the~$e$-coefficient extraction and auxiliary Kostka formula.
GPT supplied the initial proof of \cref{thm:s.001,thm:S001-two-column-coefficient},
and \cref{thm:e+.Sab2.nonzero.residue} in response to the authors' request
to advance \cref{conj:Tom.e+.Sabc}.
It strengthened \cref{lem:S001.kernel} by allowing $s=t$
and designed the strategy using
\cref{lem:path-clique.bootstrap,lem:s+.Crs,lem:s+.Drq}
to prove \cref{thm:s+.Sab2}.
GPT also checked the entire manuscript.

The authors discovered \cref{lem:alpha-omega.s},
recognized the strength of \cref{prop:s+.e-e}, constructed \cref{thm:equivalence},
and formulated \cref{cor:s+.Hall.prefix} for applications.
They established \cref{thm:omegaY.000,thm:e-.000,lem:nice-.000,thm:e+.**0,thm:s+.100,cor:e+.S001,cor:s+.010}
without GPT assistance.
Through a dozen rounds of chats with GPT,
they discovered and formulated \cref{thm:s.001}.
They completely restated \cref{thm:S001-two-column-coefficient} and rewrote the proofs of
\cref{thm:s.001,lem:clique-spider.center}
and \cref{lem:s+.Crs,lem:s+.Drq}
for readability.

Codex made manuscript revision more efficient by automating source edits.
The authors take full responsibility
for the mathematical arguments,
presentation,
accuracy,
and content of this paper.

\section*{Statements and Declarations}

\noindent\textbf{Data availability.} No datasets were generated or analysed during the current study.

\noindent\textbf{Competing interests.} The authors declare that they have no competing interests.

\bibliographystyle{abbrvnat-bysame}
\bibliography{csf}

\end{document}